\documentclass[10pt,a4paper]{amsart}

\usepackage[T1]{fontenc}
\usepackage[utf8]{inputenc}
\usepackage{lmodern}
\usepackage{amsmath,amssymb,amsfonts,mathtools}
\usepackage{mathrsfs}
\usepackage{enumitem}
\usepackage{microtype}
\usepackage[hidelinks]{hyperref}

\newcommand{\C}{\mathbb C}
\newcommand{\Ptwo}{\mathbb P^2_{\C}}
\newcommand{\Pdual}{\check{\mathbb P}^{2}_{\C}}
\newcommand{\PSL}{\mathrm{PSL}}
\newcommand{\PGL}{\mathrm{PGL}}
\newcommand{\Fl}{\mathrm{Fl}}
\newcommand{\Eq}{\mathrm{Eq}}
\newcommand{\Kul}{\mathrm{Kul}}
\newcommand{\Myr}{\mathrm{Myr}}
\newcommand{\CoG}{\mathrm{CoG}}
\newcommand{\Thinc}{\mathrm{Th}_{\mathrm{inc}}}
\newcommand{\BSPG}{\mathcal B_{\mathrm{SP},G}}
\newcommand{\BFG}{\mathcal B_{\mathrm F,G}}
\newcommand{\SP}{\mathrm{SP}}
\DeclareMathOperator{\Ker}{Ker}
\DeclareMathOperator{\Image}{Im}
\DeclareMathOperator{\Stab}{Stab}
\DeclareMathOperator{\Hom}{Hom}

\newtheorem{theorem}{Theorem}[section]
\newtheorem{proposition}[theorem]{Proposition}
\newtheorem{lemma}[theorem]{Lemma}
\newtheorem{corollary}[theorem]{Corollary}
\theoremstyle{definition}
\newtheorem{definition}[theorem]{Definition}
\theoremstyle{remark}
\newtheorem{remark}[theorem]{Remark}

\newtheorem*{maintheoremone}{Theorem 1}
\newtheorem*{maintheoremtwo}{Theorem 2}
\newtheorem*{maintheoremthree}{Theorem 3}
\newtheorem*{maintheoremfour}{Theorem 4}

\title[ASYMPTOTIC BOUNDARIES AND LIMIT SETS]{ASYMPTOTIC FLAG GEOMETRY\\OF COMPLEX KLEINIAN GROUPS IN $\Ptwo$}
\author{Waldemar Barrera}
\author{Angel Cano}
\author{Juan Pablo Navarrete}
\author{José Seade}

\subjclass[2020]{Primary 37F99; Secondary 30F40, 20H10, 57M60}

\begin{document}
\begin{abstract}
We organize several natural notions of limit set for discrete subgroups of $\PSL(3,\C)$ around a common asymptotic structure encoded by pseudo--projective degeneration and by the full or partial flag data carried by divergent sequences.  In the $\theta$-divergent case, the two projections of the full-flag limit set recover the attracting and repelling projective boundary sets, while the Myrberg limit set is the union of the limiting projective lines.  Thus the equicontinuity region is the complement of a canonical line configuration; under the usual three-line general-position hypothesis, this same configuration is the Kulkarni limit set and the equicontinuity and Kulkarni ordinary regions coincide.

This description has direct consequences for the complex geometry of the ordinary set.  We prove a local five-line criterion for complete Kobayashi hyperbolicity and hyperbolic embedding and, together with the known four-line classification and line-counting theorem, obtain the global statement that four projective lines in general position in the Kulkarni limit set already force every ordinary component to be complete Kobayashi hyperbolic and hyperbolically embedded in $\Ptwo$; these components are also taut, pseudoconvex, Stein, and domains of holomorphy.  Finally, for every $r\ge2$ we construct a nonempty open Schottky locus in $\Hom(F_r,\PSL(3,\C))$ whose groups are strongly irreducible, $\theta$-divergent, and satisfy $\nu=\infty$.  This locus is Zariski dense and has the same Kobayashi geometry.  Passing from marked representations to image groups gives the corresponding density statement in the final Zariski--Chabauty topology.  A Zassenhaus argument shows at the same time that this genericity does not extend to algebraic or strong convergence.  The resulting picture gives a projective counterpart to part of Sullivan's dictionary, with asymptotic flag data, limit sets, and the complex geometry of the ordinary region governed by the same projective configuration.
\end{abstract}

\maketitle

\section*{Introduction}
For a classical Kleinian group acting on the Riemann sphere, the limit set is the fundamental asymptotic object of the action.  It records the accumulation of orbits, separates the domain of discontinuity from the part where the essential dynamics takes place, and provides the natural boundary on which much of the geometry of the group is encoded.  This is one of the starting points of Sullivan's dictionary between Kleinian groups and holomorphic dynamics.

For discrete groups acting on higher-dimensional complex projective spaces, the picture is considerably less unified.  Already for discrete subgroups of $\PSL(3,\C)$ acting on $\Ptwo$, several notions of limit set arise naturally from different aspects of the dynamics.  The Kulkarni and Myrberg limit sets are associated with proper discontinuity and equicontinuity; the Conze--Guivarc'h limit set records attracting directions of proximal elements; and the Cartan limit set belongs naturally to the geometry of flag spaces.  These objects come from rather different constructions, and a priori there is no reason for them to be manifestations of the same geometry.

The main purpose of this paper is to show that they are organized by a common asymptotic structure.  This structure appears when one studies divergent sequences of projective transformations.  Their degenerations produce points, projective lines, and incidence relations between them, hence full or partial flag data.  The same data determine the asymptotic boundary of the group and, through its line components, control the ordinary region and its complex geometry.

The pseudo--projective viewpoint used here has a precursor in the $\mathrm{PU}(1,n)$ equicontinuity analysis of Cano--Seade~\cite{CS10}.  For discrete subgroups of $\PSL(3,\C)$, equicontinuity and pseudo--projective boundary methods were developed in~\cite{BCN11,CCL26}; the geometry and counting of projective lines in Kulkarni limit sets were studied in~\cite{BCN16}; and projective duality together with the extended Conze--Guivarc'h limit set was developed in~\cite{BGN18}.  For the complementary classification of elementary groups associated with finite general-position line configurations, see~\cite{BCNS25}.  Earlier developments of higher-dimensional complex Kleinian groups and their domains of discontinuity may be found in~\cite{SV01,SV02}; see also~\cite{CNS13} for a systematic account.  The basic degeneration trichotomy used below already appears, in a form adapted to maximal regular regions, in~\cite{CCL26}.

Our aim is not to regard these ingredients as separate constructions, but to place them in one asymptotic framework and make their compatibility explicit.  We use two projective facts from the preceding literature without reproving them: the description of the equicontinuity obstruction through kernels of pseudo--projective boundary maps~\cite{BCN11,CCL26}, and the extended Conze--Guivarc'h limit set introduced by Barrera--Gonz\'alez-Urquiza--Navarrete through the action on the dual projective plane~\cite{BGN18}.  We keep the latter terminology in its original dual sense.

There is, on the other hand, a well-developed higher-rank language for the flag-theoretic side of this picture.  Benoist introduced limit sets in full flag varieties for Zariski-dense linear semigroups~\cite{Benoist97}, and Kapovich--Leeb--Porti developed small flag limit sets and convergence dynamics for regular discrete subgroups of semisimple Lie groups~\cite{KLP18}.  In type $A_2$, the $\theta$-divergence condition used here is precisely a full-chamber regularity condition of this kind.  The point for us is that, in $\PSL(3,\C)$, the flag and projective languages can be identified very concretely: the two projections of the full-flag limit set recover the two pseudo--projective boundary sets, and the corresponding family of projective lines is exactly the Myrberg obstruction.  Under the usual three-line general-position hypothesis, it is also the Kulkarni limit set.

The elementary mechanism behind this picture is a trichotomy for divergent sequences in $\PSL(3,\C)$.  After passing to a subsequence, every such sequence has one of three possible types of degeneration: a full-flag degeneration, a point-to-line degeneration, or a line-to-point degeneration.  The sequence and its inverse record, in a symmetric way, the collapsing and concentrating directions.  We shall use the following form.

\begin{maintheoremone}
Let $G\subset \PSL(3,\C)$ be a discrete subgroup, and let $(g_n)\subset G$ be a sequence of pairwise distinct elements. Then, after passing to a subsequence, exactly one of the following mutually exclusive alternatives occurs:
\begin{enumerate}[label=(\arabic*)]
\item \emph{Full-flag case.} There exist full flags $(p^{\pm},\ell^{\pm})$ in $\Ptwo$, with $p^{\pm}\in\ell^{\pm}$, such that $g_n$ converges locally uniformly on $\Ptwo\setminus\ell^+$ to the constant map $p^-$, and $g_n^{-1}$ converges locally uniformly on $\Ptwo\setminus\ell^-$ to $p^+$. Moreover, points on $\ell^+\setminus\{p^+\}$ have $\ell^-$ as their set of destinations.
\item \emph{Point-to-line case.} There exist a point $p^+\in\Ptwo$ and a projective line $\ell^-\subset\Ptwo$ such that, outside $p^+$, the sequence $(g_n)$ converges to a pseudo--projective map with image $\ell^-$, while $(g_n^{-1})$ collapses $\Ptwo\setminus\ell^-$ to $p^+$.
\item \emph{Line-to-point case.} There exist a projective line $\ell^+\subset\Ptwo$ and a point $p^-\in\Ptwo$ such that, outside $\ell^+$, the sequence $(g_n)$ collapses to $p^-$, while $(g_n^{-1})$ converges to a pseudo--projective map with image $\ell^+$.
\end{enumerate}
\end{maintheoremone}

The significance of the trichotomy becomes apparent when one considers simultaneously the asymptotic data carried by all divergent sequences in $G$.  Pseudo--projective limits record kernels and images, while projective duality exchanges the point and line sides of the dynamics.  Our first structural result shows that, in the regular full-flag regime, the various boundary constructions are shadows of one and the same flag limit set.

\begin{maintheoremtwo}
Let $G\subset\PSL(3,\C)$ be a discrete $\theta$-divergent subgroup, let $\Lambda_\theta(G)\subset\Fl(\Ptwo)$ be its full-flag limit set, and let $\widehat\Lambda_{\CoG}(G)\subset\Pdual$ denote the extended Conze--Guivarc'h limit set of $G$ in the sense of~\cite{BGN18}. Then every element of the pseudo--projective boundary $\BSPG$ has rank one. With the canonical bidual identification, one has
\[
\widehat\Lambda_{\CoG}(\check G)=\pi_1(\Lambda_\theta(G))\subset\Ptwo,
\qquad
\widehat\Lambda_{\CoG}(G)=\pi_2(\Lambda_\theta(G))\subset\Pdual.
\]
Moreover, writing
\[
\Thinc(\Lambda_\theta(G)):=\bigcup_{(p,\ell)\in\Lambda_\theta(G)}\ell,
\]
one has the exact identities
\[
\Lambda_{\Myr}(G)=\bigcup_{[\ell]\in\widehat\Lambda_{\CoG}(G)}\ell
=\Thinc(\Lambda_\theta(G)),
\qquad
\Eq(G)=\Ptwo\setminus\Thinc(\Lambda_\theta(G)).
\]
The action of $G$ on this domain is properly discontinuous. Thus the classical normal-family obstruction is exactly the incidence Schubert thickening of the full-flag limit set in the point flag variety $\Ptwo$.

If, in addition, $\Lambda_{\Kul}(G)$ contains at least three projective lines in general position, then
\[
\Lambda_{\Kul}(G)=\Lambda_{\Myr}(G)=
\bigcup_{[\ell]\in\widehat\Lambda_{\CoG}(G)}\ell
=\Thinc(\Lambda_\theta(G)),
\]
\[
\Omega_{\Kul}(G)=\Eq(G)=\Ptwo\setminus\Thinc(\Lambda_\theta(G)).
\]
\end{maintheoremtwo}

The theorem separates what is already known in the projective theory from the comparison made here.  The description of the Myrberg set by kernels of pseudo--projective limits and the dual extended Conze--Guivarc'h construction are available from the earlier work cited above.  The $\theta$-divergent hypothesis makes every boundary map rank one and identifies the two flag projections with the attracting and repelling projective boundary sets.  Thus the Myrberg locus itself is the union of the limiting lines, with no general-position assumption.  The three-line hypothesis enters only when this same set is identified with the Kulkarni limit set.

Once the limit set has been reduced to a configuration of projective lines, its incidence geometry becomes relevant to the geometry of the complement.  Here a feature peculiar to dimension two appears.  Green's theorem gives Kobayashi hyperbolicity for the complement of five lines in general position, but completeness at a boundary point requires a supporting configuration through that point.  This leads to a local five-line criterion.  Combined with the known classification of the four-line case~\cite{BCNfour11} and the line-counting theorem~\cite{BCN16}, it yields the following global statement.

\begin{maintheoremthree}
Let $G\subset\PSL(3,\C)$ be a discrete subgroup. If $\Lambda_{\Kul}(G)$ contains at least four projective lines in general position, then every connected component of $\Omega_{\Kul}(G)$ is complete Kobayashi hyperbolic and Kobayashi hyperbolically embedded in $\Ptwo$. Moreover, every such component is taut, pseudoconvex, Stein, and a domain of holomorphy.
\end{maintheoremthree}

There are two genuinely different mechanisms behind this theorem.  When $\nu(G)=4$, one is in the hyperbolic toral case of~\cite{BCNfour11}, and the ordinary components are projectively equivalent to copies of $\mathbb H\times\mathbb H$.  When $\nu(G)>4$, the line-counting theorem forces $\nu(G)=\infty$; every boundary point lies on a limit line, and a prescribed such line can be completed to five lines in general position.  Green's theorem can then be applied locally.  We isolate this five-line criterion in Section~4 because it is the part of the argument that survives independently of the group-theoretic classification.

This geometric conclusion has a natural analogue on the holomorphic side of Sullivan's dictionary.  In the work of Forn\ae ss and Sibony, sufficiently rich incidence properties of postcritical sets lead to strong Kobayashi geometry for their complements~\cite{FS94,FS95}.  Here the postcritical configuration is replaced by the family of projective lines produced by the asymptotic dynamics.  It is therefore natural to ask whether this general-position regime is exceptional or whether it occurs abundantly among projective Kleinian groups.

The answer is provided by Schottky dynamics.  One may choose a projective Schottky configuration whose attracting and repelling flags are in general position and whose Cartan directions stay in the interior of the Weyl chamber.  For sufficiently large translation lengths the resulting free group is $\theta$-divergent.  The ping--pong and regularity conditions are strict and hence persist under small perturbations.  This produces a nonempty open set in the usual complex topology of the marked representation variety.  Since $\Hom(F_r,\PSL(3,\C))$ is irreducible, every such open set is Zariski dense.

For later use, let $\mathcal K_r$ denote the locus of representations $\rho\in\Hom(F_r,\PSL(3,\C))$ whose image is a non-virtually-cyclic complex Kleinian group.

\begin{maintheoremfour}
For every $r\ge2$ there exists a nonempty subset
\[
\mathcal U_r\subset\Hom(F_r,\PSL(3,\C)),
\]
open in the usual complex topology, such that, for every $\rho\in\mathcal U_r$, the representation $\rho$ is injective and $\Gamma_\rho=\rho(F_r)$ is a discrete, strongly irreducible, $\theta$-divergent Schottky-type complex Kleinian group satisfying
\[
\nu(\Gamma_\rho)=\infty.
\]
Moreover, every point of $\Lambda_{\Kul}(\Gamma_\rho)$ lies on a limit line that can be completed, inside $\Lambda_{\Kul}(\Gamma_\rho)$, to five lines in general position. Consequently every connected component of $\Omega_{\Kul}(\Gamma_\rho)$ is complete Kobayashi hyperbolic and Kobayashi hyperbolically embedded in $\Ptwo$, and hence is taut, pseudoconvex, Stein, and a domain of holomorphy.

The set $\mathcal U_r$ is Zariski dense in $\Hom(F_r,\PSL(3,\C))$. In particular, the loci in $\mathcal K_r$ consisting of strongly irreducible representations, of representations with $\nu=\infty$, and of representations with the preceding Kobayashi geometry are all Zariski dense in $\mathcal K_r$.
\end{maintheoremfour}

The theorem gives the marked form of the genericity statement.  There is also a natural unmarked version.  We put on finitely generated image groups the final topology induced by the forgetful maps $\rho\mapsto\rho(F_r)$ and call it the Zariski--Chabauty topology.  For a fixed bound on the number of generators, the groups with $\nu=\infty$ are dense in this topology, and the approximating groups may be chosen from the same Schottky locus.  In particular, their ordinary components have the geometry described above.

It is important that this last statement is topological in the precise sense just specified.  Zariski density in the marked space, and the corresponding final topology after forgetting the marking, are much weaker than density for the usual analytic notions of convergence of discrete groups.  A Zassenhaus-neighborhood argument gives an ordinary neighborhood of a rank-two abelian complex Kleinian group which does not meet the Schottky locus.  Thus the marked statement does not extend to algebraic convergence; after forgetting the marking, the same example also obstructs strong convergence.  We make no analogous claim here for geometric convergence alone.

Taken together, these results give a coherent picture.  Divergent sequences produce asymptotic flag data; this data organizes the different notions of limit set; the resulting projective-line configuration determines the ordinary region and, in general position, its complete intrinsic hyperbolic geometry; and the configurations for which this picture holds occur on a Zariski-dense marked locus and densely in the corresponding unmarked Zariski--Chabauty space.

From this point of view, the analogy with the classical theory is not merely formal.  For classical Kleinian groups, one limit set is at once an asymptotic boundary, the locus on which the essential dynamics accumulates, and the complement of the domain of discontinuity.  In complex projective dimension two these roles are distributed among several objects.  The pseudo--projective boundary and its flag data place them in a common asymptotic geometry: the Kulkarni limit set and the ordinary set retain the familiar boundary/discontinuity-space dichotomy, while the Conze--Guivarc'h and Cartan loci record complementary attracting and flag data.  In this sense, the results above provide a projective counterpart to part of Sullivan's dictionary.

\section{Projective dynamics and preliminaries}
We fix notation and recall the projective-dynamical objects used throughout the paper, with emphasis on the structures that encode degenerations of projective transformations.

\subsection{Projective space}
Let
\[
\Ptwo=(\C^3\setminus\{0\})/\C^*
\]
be the complex projective plane. Its group of holomorphic automorphisms is
\[
\operatorname{Aut}(\Ptwo)\simeq \PSL(3,\C).
\]
A projective line $\ell\subset\Ptwo$ is the projectivization of a $2$-dimensional linear subspace of $\C^3$. The space of projective lines is naturally identified with the dual projective plane $\Pdual$.

We recall the classification of elements of $\PSL(3,\C)$; see~\cite{CNS13}. It will be used to describe asymptotic behavior and, in particular, the role of strongly loxodromic elements.

Let $g\in\PSL(3,\C)$, and choose a lift $\widetilde g\in\mathrm{SL}(3,\C)$ such that $[\widetilde g]=g$. One says that $g$ is elliptic if $\widetilde g$ is diagonalizable and all its eigenvalues have modulus $1$, parabolic if $\widetilde g$ is not diagonalizable and all its eigenvalues have modulus $1$, and loxodromic otherwise. Equivalent dynamical and algebraic descriptions may be found in~\cite{CNS13}.

Up to conjugation, loxodromic elements fall into four classes described in~\cite{CNS13}: complex homotheties, screws, loxoparabolic elements, and strongly loxodromic elements. The last class plays a distinguished role here: strongly loxodromic elements are precisely the biproximal elements and govern the generic asymptotic behavior considered below.

\begin{definition}
A matrix $\widetilde g\in\mathrm{GL}(3,\C)$ is proximal if it has an eigenvalue $t_0\in\C$ such that
\[
|t_0|>|t|
\]
for every other eigenvalue $t$ of $\widetilde g$, and $t_0$ has algebraic multiplicity one. An eigenvector corresponding to $t_0$ is called a dominant eigenvector. An element $g\in\PSL(3,\C)$ is proximal if some lift $\widetilde g\in\mathrm{SL}(3,\C)$ is proximal. A point $v\in\Ptwo$ is called dominant for $g$ if some lift of $v$ to $\C^3$ is a dominant eigenvector of such a lift.
\end{definition}

Proximal elements play a central role in the asymptotic description of the dynamics. A proximal element of $\PSL(3,\C)$ has a unique attracting fixed point in $\Ptwo$.

\begin{remark}
For $\PSL(3,\C)$, the classification in~\cite{CNS13} implies that every proximal element is loxodromic. Conversely, if $g\in\PSL(3,\C)$ is loxodromic, then either $g$ or $g^{-1}$ is proximal. In particular, strongly loxodromic elements are proximal in both forward and backward directions.
\end{remark}

\begin{remark}
By the classification of loxodromic elements in $\PSL(3,\C)$, the asymptotic dynamics of a loxodromic element $g$ fall into two cases.

In the first case, $g$ admits two distinguished invariant projective lines $\ell_g^+,\ell_g^-\subset\Ptwo$. The line $\ell_g^+$ is attracting for $g$ and repelling for $g^{-1}$, while $\ell_g^-$ is repelling for $g$ and attracting for $g^{-1}$. This occurs, for instance, for strongly loxodromic elements.

In the second case, $g$ admits a single distinguished invariant projective line $\ell_g$, together with a fixed point $p_g\notin\ell_g$. In this situation, one of $\ell_g$ or $p_g$ is attracting, and the other repelling, with the roles reversed for $g^{-1}$. Thus, there is only one dynamically distinguished line: if it is repelling (resp. attracting), we denote it by $\ell_g^-$ (resp. $\ell_g^+$).

These lines are naturally described on the dual projective plane. Via the identification
\[
\mathbb P(\wedge^2\C^3)\simeq(\Ptwo)^*,
\]
a projective line in $\Ptwo$ corresponds to a point of $\mathbb P(\wedge^2\C^3)$. When $\wedge^2g$ is proximal, its dominant fixed point represents the attracting line of $g$. When $\wedge^2g^{-1}$ is proximal, its dominant fixed point represents the repelling line of $g$. Accordingly, throughout the paper we use the notation $\ell_g^-$ only when the repelling object of $g$ is a projective line. This convention is compatible with the dual Conze--Guivarc'h viewpoint used below.
\end{remark}

\subsection{Pseudo--projective maps}
The space of pseudo--projective maps provides a natural compactification of $\PSL(3,\C)$ that records degenerations of projective transformations. It will be our main tool for describing asymptotic behavior.

Let $M_{3\times3}(\C)$ be the space of complex $3\times3$ matrices, and define
\[
\SP(3,\C):=(M_{3\times3}(\C)\setminus\{0\})/\C^*.
\]
This space is naturally identified with $\mathbb P^8_{\C}$ and contains $\PSL(3,\C)$ as an open dense subset.

Each element $S\in\SP(3,\C)$ induces a partially defined map
\[
S:\Ptwo\setminus\Ker(S)\longrightarrow\Ptwo,
\]
where $\Ker(S)$ is a projective subspace (possibly empty), and we define
\[
\Image(S):=S(\Ptwo\setminus\Ker(S)).
\]
We use the following compactness property for sequences of projective transformations.

\begin{proposition}
Let $(g_n)\subset\PSL(3,\C)$ be a sequence of pairwise distinct elements. Then, after passing to a subsequence, there exists $\rho\in\SP(3,\C)$ such that
\[
g_n\longrightarrow\rho
\]
in $\SP(3,\C)$, and $g_n$ converges uniformly on compact subsets of $\Ptwo\setminus\Ker(\rho)$ to the map induced by $\rho$.
\end{proposition}

\subsection{Limit sets}
Let $G\subset\PSL(3,\C)$ be a discrete subgroup. We recall the notions of limit set that will later be compared through the asymptotic boundary. Throughout the paper, a complex Kleinian group means a discrete subgroup $G\subset\PSL(3,\C)$ with nonempty Kulkarni discontinuity region $\Omega_{\Kul}(G)$.
\begin{itemize}
\item The Kulkarni limit set $\Lambda_{\Kul}(G)$ is defined as in~\cite{Kul78}. Its complement
\[
\Omega_{\Kul}(G):=\Ptwo\setminus\Lambda_{\Kul}(G)
\]
is the Kulkarni discontinuity region; the action of $G$ on $\Omega_{\Kul}(G)$ is properly discontinuous.
\item The Myrberg limit set $\Lambda_{\Myr}(G)$ is defined by
\[
\Lambda_{\Myr}(G):=\Ptwo\setminus\Eq(G),
\]
where $\Eq(G)$ is the equicontinuity region; for the modern projective formulation used here see~\cite{BCN11}, while the terminology is historically associated with~\cite{Myrberg25}.
\item In the standard strongly irreducible setting, the Conze--Guivarc'h limit set $\Lambda_{\CoG}(G)$ is the closure of the attracting fixed points of proximal elements, whenever such elements exist; see~\cite{CG00}.
\item Following~\cite{BGN18}, the extended Conze--Guivarc'h limit set $\widehat\Lambda_{\CoG}(G)$ is defined in the dual projective plane $\Pdual$ from convergence of the induced dual action on a nonempty open set. This definition makes sense without assuming strong irreducibility or the existence of proximal elements. Under the standard Conze--Guivarc'h hypotheses of strong irreducibility and proximality, the inclusions comparing the classical and extended sets in Lemma~3.1 become equalities after the canonical bidual identification; see the remark following that lemma. We keep this dual-space convention throughout.
\item The full-flag (or Cartan) limit set $\Lambda_\theta(G)$ records asymptotic flag data via the Cartan projection.
\end{itemize}

\subsection{Compactifications}
We next recall three compactifications that encode degenerations of projective transformations from complementary perspectives.

\medskip\noindent\emph{Pseudo--projective compactification.}
This is the space $\SP(3,\C)$ defined above, which records degenerations in terms of kernel and image data.

\medskip\noindent\emph{Pointwise (Furstenberg-type) compactification.}
We use the term Furstenberg map for the pointwise degeneration model below. The terminology emphasizes its flag-boundary character; throughout the paper, it refers specifically to the following three-type definition.

\begin{definition}
A map
\[
\rho:\Ptwo\longrightarrow\Ptwo
\]
is called a Furstenberg map if either $\rho\in\PSL(3,\C)$, or else one of the following holds:
\begin{enumerate}[label=(\arabic*)]
\item There exists an incident pair $p\in\ell$, where $p$ is a point and $\ell$ is a line in $\Ptwo$, and points $q_0,q_1,q_2\in\Ptwo$, not necessarily distinct, such that
\[
\rho(p)=q_0,\qquad \rho|_{\ell\setminus\{p\}}\equiv q_1,\qquad \rho|_{\Ptwo\setminus\ell}\equiv q_2.
\]
\item There exist a point $p\in\Ptwo$, a line $\ell\subset\Ptwo$, a pseudo--projective map
\[
P:\Ptwo\setminus\{p\}\longrightarrow\ell
\]
induced by a rank-$2$ linear map with kernel $p$, and a point $q\in\Ptwo$, such that
\[
\rho|_{\Ptwo\setminus\{p\}}=P,\qquad \rho(p)=q.
\]
\item There exist projective lines $\ell_1,\ell_2\subset\Ptwo$, a projective isomorphism
\[
f:\ell_1\longrightarrow\ell_2,
\]
and a point $q\in\Ptwo$, such that
\[
\rho|_{\ell_1}=f,\qquad \rho|_{\Ptwo\setminus\ell_1}\equiv q.
\]
\end{enumerate}
\end{definition}

The corresponding compactness statement is the following.

\begin{lemma}
Let $(g_m)\subset\PSL(3,\C)$ be a sequence of pairwise distinct elements. Then, after passing to a subsequence, there exists a Furstenberg map
\[
\rho:\Ptwo\longrightarrow\Ptwo
\]
such that $g_m\to\rho$ pointwise on $\Ptwo$.
\end{lemma}
\begin{proof}
Choose matrix representatives $A_m\in\mathrm{SL}(3,\C)$ and scalars $c_m\in\C^*$ so that $\|c_mA_m\|=1$. After passing to a subsequence,
\[
c_mA_m\longrightarrow A\ne0
\]
in $M_{3\times3}(\C)$. If $\operatorname{rank}A=3$, the induced projective transformations converge to the element of $\PSL(3,\C)$ represented by $A$.

Suppose that $\operatorname{rank}A=2$. Then $\mathbb P(\Ker A)$ is a point $p$, and on $\Ptwo\setminus\{p\}$ the sequence converges to the rank-two pseudo--projective map induced by $A$. Since $A_m(p)$ is defined for every $m$ and $\Ptwo$ is compact, we may pass to a further subsequence for which $A_m(p)$ converges to a point $q$. The resulting pointwise limit is therefore of type~(2).

It remains to consider $\operatorname{rank}A=1$. Put $L:=\mathbb P(\Ker A)$, a projective line. On $\Ptwo\setminus L$ the sequence converges to the constant point $\mathbb P(\Im A)$. To analyze the restriction to $L$, choose nonzero scalars $d_m$ so that
\[
\bigl\|d_m(c_mA_m)|_{\Ker A}\bigr\|=1.
\]
After passing to a subsequence,
\[
d_m(c_mA_m)|_{\Ker A}\longrightarrow B:\Ker A\longrightarrow\C^3,\qquad B\ne0.
\]
If $\operatorname{rank}B=2$, the projectivization of $B$ is a projective isomorphism from $L$ onto the projective line $\mathbb P(\Im B)$; together with the constant behavior off $L$, this gives type~(3). If $\operatorname{rank}B=1$, then $\mathbb P(\Ker B)=\{p\}$ is a point of $L$ and the restriction to $L\setminus\{p\}$ converges to the constant point $\mathbb P(\Im B)$. Finally, compactness of $\Ptwo$ allows us to pass to a further subsequence for which $A_m(p)$ converges. Thus the pointwise limit is of type~(1). In all cases we obtain a Furstenberg map after passing to a subsequence.
\end{proof}

\begin{definition}
Let $\mathcal F(\Ptwo)$ denote the set of Furstenberg maps of $\Ptwo$. For $\rho\in\mathcal F(\Ptwo)$ we write
\[
\Image(\rho):=\rho(\Ptwo)
\]
for its pointwise image. For a subgroup $G\subset\PSL(3,\C)$, the Furstenberg boundary $\BFG$ is the set of all Furstenberg maps arising as pointwise limits along sequences of pairwise distinct elements of $G$. In particular, $\BFG=\varnothing$ when $G$ is finite.
\end{definition}

\medskip\noindent\emph{Wonderful compactification.}
We recall only the features that will be used later. For the general theory of wonderful compactifications see~\cite{DCP86}; for the complete-collineation realization and its determinantal blow-up description see~\cite{Vain84}. Since $\PSL(3,\C)\simeq\PGL(3,\C)$, we do not distinguish between these groups.

Let $B\subset\PSL(3,\C)$ be a Borel subgroup, and let $P_{\alpha_1},P_{\alpha_2}\subset\PSL(3,\C)$ be the maximal parabolic subgroups corresponding to the simple roots $\alpha_1,\alpha_2$. Then
\[
\PSL(3,\C)/B\simeq\Fl(\Ptwo),\qquad
\PSL(3,\C)/P_{\alpha_1}\simeq\Ptwo,\qquad
\PSL(3,\C)/P_{\alpha_2}\simeq\Pdual.
\]
Consider the projective space
\[
\SP(3,\C)=\mathbb P(M_{3\times3}(\C)),
\]
which contains $\PGL(3,\C)$ as the open subset of invertible classes. Its rank-one locus is naturally identified with
\[
\Sigma\simeq\Ptwo\times\Pdual,
\]
since a rank-one endomorphism is determined by its image point and kernel line.

In rank three, the complete-collineation model realizes the wonderful compactification of $\PGL(3,\C)$ by blowing up $\SP(3,\C)$ along the rank-one locus $\Sigma$; see~\cite{Vain84}. This blow-up separates kernel and image data and resolves their incidence.

The boundary
\[
\PGL(3,\C)^{\mathrm{wond}}\setminus\PGL(3,\C)
\]
is a simple normal crossings divisor with two irreducible components, corresponding to the two simple roots. We denote them by $D_{\alpha_1}$ and $D_{\alpha_2}$, using the convention that, in the positive Weyl-chamber chart used below, the associated normal parameters are
\[
t_i=e^{-\delta_i},\qquad i=1,2,
\]
where $\delta_1=\lambda_1-\lambda_2$ and $\delta_2=\lambda_2-\lambda_3$ are the two simple Cartan gaps. Thus $D_{\alpha_i}=\{t_i=0\}$ in this chart. The closed stratum is isomorphic to
\[
(\PGL(3,\C)/B)\times(\PGL(3,\C)/B).
\]
These compactifications encode complementary aspects of a degeneration. The pseudo--projective compactification remains the primary framework for the asymptotic analysis below.

\section{Asymptotic flag geometry}
We now analyze divergent sequences in $\PSL(3,\C)$ and prove Theorem~1. The Cartan decomposition reduces the problem to diagonal sequences, whose relative rates of divergence determine the associated full or partial flag data.

\subsection{Cartan decomposition}
Let $g\in\mathrm{SL}(3,\C)$. By the Cartan decomposition, one can write
\[
g=k_1ak_2,
\]
where $k_1,k_2\in\mathrm{SU}(3)$ and
\[
a=\operatorname{diag}(e^{\lambda_1},e^{\lambda_2},e^{\lambda_3}),
\qquad
\lambda_1\ge\lambda_2\ge\lambda_3,
\qquad
\lambda_1+\lambda_2+\lambda_3=0.
\]
The vector $(\lambda_1,\lambda_2,\lambda_3)$ is uniquely determined and defines the Cartan projection.

Let $(g_n)\subset\PSL(3,\C)$ be a sequence of pairwise distinct elements. After choosing lifts in $\mathrm{SL}(3,\C)$ and passing to a subsequence, we may assume that
\[
g_n=k_na_nk_n',
\qquad
a_n=\operatorname{diag}(e^{\lambda_1^{(n)}},e^{\lambda_2^{(n)}},e^{\lambda_3^{(n)}}),
\]
and
\[
\lambda_1^{(n)}\ge\lambda_2^{(n)}\ge\lambda_3^{(n)},
\qquad
\lambda_1^{(n)}+\lambda_2^{(n)}+\lambda_3^{(n)}=0.
\]
The sequence $(g_n)$ is divergent if and only if
\[
\max_i|\lambda_i^{(n)}|\longrightarrow+\infty.
\]

\subsection{Reduction to diagonal dynamics}
After passing to a subsequence, we may assume that $(k_n)$ and $(k_n')$ converge. After composing on the left and right with the limiting unitary transformations, the asymptotic analysis reduces to the diagonal sequence $(a_n)$.

The asymptotic behavior is governed by the gaps
\[
\lambda_1^{(n)}-\lambda_2^{(n)},
\qquad
\lambda_2^{(n)}-\lambda_3^{(n)}.
\]
For $g\in\PSL(3,\C)$, write
\[
\delta_1(g):=\lambda_1(g)-\lambda_2(g),
\qquad
\delta_2(g):=\lambda_2(g)-\lambda_3(g).
\]
A subgroup $G\subset\PSL(3,\C)$ is called $\theta$-divergent if, for every sequence $(g_n)$ of pairwise distinct elements of $G$,
\[
\delta_1(g_n)\longrightarrow+\infty,
\qquad
\delta_2(g_n)\longrightarrow+\infty.
\]
If both gaps of $g=k_1ak_2$ are positive, its left singular flag is well-defined and we write
\[
U(g):=k_1\bigl([e_1],\mathbb P\langle e_1,e_2\rangle\bigr)\in\Fl(\Ptwo).
\]
For a $\theta$-divergent group, all but finitely many terms of every sequence of pairwise distinct elements have positive gaps. Its full-flag Cartan limit set is
\[
\Lambda_\theta(G):=
\left\{
F\in\Fl(\Ptwo):
\begin{array}{c}
\text{there are pairwise distinct }g_n\in G\text{ with}\\
U(g_n)\longrightarrow F
\end{array}
\right\}.
\]
Equivalently, $\Lambda_\theta(G)$ is the accumulation set of the left singular flags. Since $\Fl(\Ptwo)$ is compact, this set is closed.

In higher-rank terminology, this is the full-chamber regularity condition for the type-$A_2$ symmetric space associated with $\PSL(3,\C)$: every sequence escaping to infinity moves away from both walls of the positive Weyl chamber. In the language of Kapovich--Leeb--Porti~\cite[Definition~6.9, Corollary~6.17]{KLP18}, such a group is $\sigma_{\mathrm{mod}}$-regular (equivalently, $\sigma_{\mathrm{mod}}$-convergence), and its small full-flag limit set agrees with the accumulation set $\Lambda_\theta(G)$ defined here. Benoist's full-flag limit set for Zariski-dense linear groups is an earlier version of the same regular boundary object~\cite{Benoist97}. We retain the present notation because the proofs below use the two simple Cartan gaps explicitly and because we need to compare this flag limit set with the projective Myrberg, Kulkarni, and Conze--Guivarc'h constructions.

To describe the accumulation pattern of images of nearby points, we use the following notion.

\begin{definition}
Let $(g_m)\subset\PSL(3,\C)$ and $x\in\Ptwo$. The set of destinations of $x$ is
\[
D_{(g_m)}(x)=
\left\{
y\ \middle|\
\begin{array}{c}
\exists\,x_m\to x,\ \exists\,m_j\to\infty\text{ such that}\\
g_{m_j}(x_{m_j})\to y
\end{array}
\right\}.
\]
\end{definition}

We also use the following continuity property.

\begin{lemma}
If $g_m\to\rho\in\SP(3,\C)$ and $h_m\to\varrho\in\PSL(3,\C)$, then
\[
g_m\circ h_m\to\rho\circ\varrho,
\qquad
h_m\circ g_m\to\varrho\circ\rho.
\]
\end{lemma}
\begin{proof}
Choose representatives so that $t_m\widetilde g_m\to A$ and $\widetilde h_m\to B$ with $B$ invertible. Then
\[
t_m(\widetilde g_m\widetilde h_m)\to AB,
\qquad
t_m(\widetilde h_m\widetilde g_m)\to BA,
\]
which proves the claim.
\end{proof}

\subsection{Proof of Theorem 1}
We prove the asymptotic trichotomy.
\begin{proof}
Choose lifts $\widetilde g_m\in\mathrm{SL}(3,\C)$ and write
\[
\widetilde g_m=\kappa_m^+D_m\kappa_m^-,
\qquad
D_m=\operatorname{diag}(e^{\alpha_m},e^{\beta_m},e^{\gamma_m}),
\]
with
\[
\alpha_m\ge\beta_m\ge\gamma_m,
\qquad
\alpha_m+\beta_m+\gamma_m=0.
\]
Since $(g_m)$ is divergent, we have
\[
\max\{|\alpha_m|,|\beta_m|,|\gamma_m|\}\longrightarrow\infty,
\]
and, after passing to a subsequence, we may assume that $\alpha_m\to+\infty$ and $\gamma_m\to-\infty$.

Passing to a subsequence,
\[
\kappa_m^\pm\to\kappa^\pm,
\qquad
\alpha_m-\beta_m\to a\in[0,\infty],
\qquad
\beta_m-\gamma_m\to b\in[0,\infty].
\]
Since
\[
\alpha_m-\gamma_m=(\alpha_m-\beta_m)+(\beta_m-\gamma_m)\longrightarrow+\infty,
\]
the two gaps cannot both remain bounded. Hence at least one of $a,b$ is infinite. This yields three cases:
\[
\text{(I) }a=b=\infty,
\qquad
\text{(II) }a<\infty,\ b=\infty,
\qquad
\text{(III) }a=\infty,\ b<\infty.
\]

\emph{Case (I): full flags.} We have
\[
e^{-\alpha_m}D_m\to\operatorname{diag}(1,0,0),
\qquad
e^{\gamma_m}D_m^{-1}\to\operatorname{diag}(0,0,1).
\]
Thus
\[
[D_m]\to\upsilon^+,
\qquad
[D_m^{-1}]\to\upsilon^-,
\]
and
\[
g_m\to\rho^+,
\qquad
g_m^{-1}\to\rho^-.
\]
These limits determine flags $(p^\pm,\ell^\pm)$. Uniform convergence follows away from kernels.

We now verify the statement about destinations. It is enough to work in the diagonal model. In these coordinates,
\[
\ell^+=\{z_1=0\},
\qquad
p^+=[0:0:1],
\qquad
\ell^-=\{z_3=0\}.
\]
Let
\[
x=[0:u:v]\in\ell^+\setminus\{p^+\},
\]
so that $u\ne0$, and let
\[
x_m=[\xi_m:u_m:v_m]\longrightarrow x.
\]
Then
\[
D_mx_m=[e^{\alpha_m}\xi_m:e^{\beta_m}u_m:e^{\gamma_m}v_m].
\]
Dividing by $e^{\beta_m}u_m$, we obtain
\[
D_mx_m=
\left[
e^{\alpha_m-\beta_m}\frac{\xi_m}{u_m}:1:
e^{\gamma_m-\beta_m}\frac{v_m}{u_m}
\right].
\]
Since $\beta_m-\gamma_m\to+\infty$, the third coordinate tends to $0$. On the other hand, because $\alpha_m-\beta_m\to+\infty$, by choosing $\xi_m\to0$ appropriately the first coordinate can be made to converge to any prescribed value in $\C$, and it can also diverge to infinity. Therefore every point of the line
\[
\ell^-=\{z_3=0\}
\]
arises as a destination of $x$, and no destination can lie outside $\ell^-$. Hence
\[
D_{(g_m)}(x)=\ell^-.
\]
The corresponding statement for $g_m^{-1}$ follows by symmetry.

\emph{Case (II): point-to-line.} We obtain limits
\[
g_m\to\rho^+,
\qquad
g_m^{-1}\to\rho^-,
\]
with
\[
\Ker(\rho^+)=\{p^+\},
\qquad
\Image(\rho^+)=\ell^-,
\]
\[
\Ker(\rho^-)=\ell^-,
\qquad
\Image(\rho^-)=\{p^+\}.
\]
Uniform convergence holds on compact subsets away from the kernel. If $K\subset\Ptwo\setminus\{p^+\}$ is compact, then the accumulation set of $\{g_m(x):x\in K\}$ is exactly $\rho^+(K)$.

\emph{Case (III): line-to-point.} This case is symmetric. One obtains
\[
\Ker(\rho^+)=\ell^+,
\qquad
\Image(\rho^+)=\{p^-\},
\]
\[
\Ker(\rho^-)=\{p^-\},
\qquad
\Image(\rho^-)=\ell^+.
\]
Uniform convergence again holds away from the kernels, and images collapse to $p^-$. The three cases are mutually exclusive and exhaustive.
\end{proof}

\subsection{Compatibility between compactifications}
We compare the three compactifications through the asymptotic data carried by a fixed divergent sequence. They are not being identified: the pseudo--projective compactification records the principal kernel and image; the pointwise Furstenberg-type compactification retains the corresponding pointwise behavior and may refine the exceptional locus; and the wonderful compactification records the parabolic type of the degeneration. Their common datum is the full or partial flag supplied by Theorem~1.

Let $D_{\alpha_1}$ and $D_{\alpha_2}$ denote the two irreducible boundary divisors of the wonderful compactification, indexed by the simple roots, and write
\[
D_{\alpha_i}^{\circ}:=D_{\alpha_i}\setminus(D_{\alpha_1}\cap D_{\alpha_2}).
\]
Thus $D_{\alpha_1}\cap D_{\alpha_2}$ is the unique closed $\PSL(3,\C)\times\PSL(3,\C)$-orbit in the boundary.

\begin{proposition}
Let $G\subset\PSL(3,\C)$ be a discrete subgroup, and let $(g_n)\subset G$ be a sequence of pairwise distinct elements. Suppose that, along the same subsequence,
\[
g_n\to\tau^+,
\qquad
g_n^{-1}\to\tau^-
\quad\text{in }\SP(3,\C),
\]
\[
g_n\to\gamma^+,
\qquad
g_n^{-1}\to\gamma^-
\quad\text{pointwise on }\Ptwo,
\]
and
\[
g_n\to\xi^+,
\qquad
g_n^{-1}\to\xi^-
\quad\text{in the wonderful compactification}.
\]
Let $F^+$ and $F^-$ denote, respectively, the source-side and target-side full or partial flag data supplied by Theorem~1. Thus, in the full-flag case, $F^+$ is the repelling/source flag and $F^-$ is the attracting/target flag. Then exactly one of the following mutually exclusive alternatives holds.
\begin{enumerate}[label=(\arabic*)]
\item \emph{Full-flag case.} There are full flags
\[
F^+=(p^+,\ell^+),
\qquad
F^-=(p^-,\ell^-),
\qquad
p^\pm\in\ell^\pm,
\]
and
\[
\Ker(\tau^+)=\ell^+,
\quad
\Image(\tau^+)=\{p^-\},
\quad
\Ker(\tau^-)=\ell^-,
\quad
\Image(\tau^-)=\{p^+\}.
\]
Consequently,
\[
\gamma^+|_{\Ptwo\setminus\ell^+}\equiv p^- ,
\qquad
\gamma^-|_{\Ptwo\setminus\ell^-}\equiv p^+.
\]
Any remaining pointwise behavior is confined to the exceptional lines and constitutes the finer Furstenberg information attached to the same full-flag degeneration. Moreover,
\[
\xi^+,\xi^-\in D_{\alpha_1}\cap D_{\alpha_2}.
\]

\item \emph{Point-to-line case.} One has
\[
F^+=p^+\in\Ptwo,
\qquad
F^-=\ell^-\subset\Ptwo,
\]
and
\[
\Ker(\tau^+)=\{p^+\},
\quad
\Image(\tau^+)=\ell^-,
\quad
\Ker(\tau^-)=\ell^-,
\quad
\Image(\tau^-)=\{p^+\}.
\]
On $\Ptwo\setminus\{p^+\}$ the Furstenberg limit $\gamma^+$ agrees with the rank-two pseudo--projective map induced by $\tau^+$; in particular it is of type~(2) in Definition~1.5, with source point $p^+$ and image line $\ell^-$. On the inverse side,
\[
\gamma^-|_{\Ptwo\setminus\ell^-}\equiv p^+,
\]
so any further pointwise refinement of $\gamma^-$ is supported on the exceptional line $\ell^-$. The wonderful limits distinguish the two maximal parabolic types:
\[
\xi^+\in D_{\alpha_2}^{\circ},
\qquad
\xi^-\in D_{\alpha_1}^{\circ}.
\]

\item \emph{Line-to-point case.} One has
\[
F^+=\ell^+\subset\Ptwo,
\qquad
F^-=p^-\in\Ptwo,
\]
and
\[
\Ker(\tau^+)=\ell^+,
\quad
\Image(\tau^+)=\{p^-\},
\quad
\Ker(\tau^-)=\{p^-\},
\quad
\Image(\tau^-)=\ell^+.
\]
Here
\[
\gamma^+|_{\Ptwo\setminus\ell^+}\equiv p^- ,
\]
while on $\Ptwo\setminus\{p^-\}$ the limit $\gamma^-$ agrees with the rank-two pseudo--projective map induced by $\tau^-$ and is therefore of type~(2), with source point $p^-$ and image line $\ell^+$. Any finer pointwise behavior of $\gamma^+$ is confined to $\ell^+$. In the wonderful compactification the two divisor types are interchanged:
\[
\xi^+\in D_{\alpha_1}^{\circ},
\qquad
\xi^-\in D_{\alpha_2}^{\circ}.
\]
\end{enumerate}
\end{proposition}

\begin{proof}
Choose lifts $\widetilde g_n\in\mathrm{SL}(3,\C)$ and, after passing to a further subsequence without changing any of the limits in the statement, write
\[
\widetilde g_n=\kappa_n^+D_n\kappa_n^- ,
\qquad
D_n=\operatorname{diag}(e^{\lambda_{1,n}},e^{\lambda_{2,n}},e^{\lambda_{3,n}}),
\]
where
\[
\lambda_{1,n}\ge\lambda_{2,n}\ge\lambda_{3,n},
\qquad
\lambda_{1,n}+\lambda_{2,n}+\lambda_{3,n}=0,
\qquad
\kappa_n^\pm\longrightarrow\kappa^\pm\in\mathrm{SU}(3).
\]
Set
\[
\delta_1(n):=\lambda_{1,n}-\lambda_{2,n},
\qquad
\delta_2(n):=\lambda_{2,n}-\lambda_{3,n}.
\]
Because $G$ is discrete and the $g_n$ are pairwise distinct, $(g_n)$ leaves every compact subset of $\PSL(3,\C)$. Hence its Cartan projection is unbounded. After passing to a subsequence, each $\delta_i(n)$ either remains bounded and converges in $[0,+\infty)$ or tends to $+\infty$; the two gaps cannot both remain bounded. We are therefore in exactly one of the three gap regimes of Theorem~1.

We first recover the pseudo--projective data directly from the gaps. The normalization by the largest singular value gives
\[
e^{-\lambda_{1,n}}D_n=
\operatorname{diag}\bigl(1,e^{-\delta_1(n)},e^{-(\delta_1(n)+\delta_2(n))}\bigr).
\]
For the inverse sequence, normalization by its largest singular value gives
\[
e^{\lambda_{3,n}}D_n^{-1}=
\operatorname{diag}\bigl(e^{-(\delta_1(n)+\delta_2(n))},e^{-\delta_2(n)},1\bigr).
\]
If both gaps tend to $+\infty$, both normalized matrices above converge to rank-one matrices. In terms of the limiting unitary factors,
\[
p^-=[\kappa^+e_1],
\qquad
\ell^+=\mathbb P\bigl((\kappa^-)^{-1}\langle e_2,e_3\rangle\bigr),
\]
\[
p^+=[(\kappa^-)^{-1}e_3],
\qquad
\ell^-=\mathbb P\bigl(\kappa^+\langle e_1,e_2\rangle\bigr).
\]
These formulas give
\[
\Ker(\tau^+)=\ell^+,
\quad
\Image(\tau^+)=\{p^-\},
\quad
\Ker(\tau^-)=\ell^-,
\quad
\Image(\tau^-)=\{p^+\}.
\]
They also make the incidence relations transparent: $p^+\in\ell^+$ and $p^-\in\ell^-$, because $e_3\in\langle e_2,e_3\rangle$ and $e_1\in\langle e_1,e_2\rangle$. With the Cartan convention used below, the left singular flags satisfy
\[
U(g_n)\longrightarrow F^-=(p^-,\ell^-),
\qquad
U(g_n^{-1})\longrightarrow F^+=(p^+,\ell^+).
\]
This fixes the $+/-$ convention for the remainder of the comparison.

Suppose next that $\delta_1(n)\to d<+\infty$ and $\delta_2(n)\to+\infty$. Then the forward normalization converges to
\[
\operatorname{diag}(1,e^{-d},0),
\]
which has rank two, whereas the inverse normalization converges to $\operatorname{diag}(0,0,1)$, which has rank one. After conjugating by the limiting unitary factors this is exactly
\[
\Ker(\tau^+)=\{p^+\},
\quad
\Image(\tau^+)=\ell^-,
\quad
\Ker(\tau^-)=\ell^-,
\quad
\Image(\tau^-)=\{p^+\}.
\]
This is the point-to-line alternative. The case $\delta_1(n)\to+\infty$ and $\delta_2(n)\to d<+\infty$ is obtained by the same calculation, or equivalently by applying the preceding argument to $(g_n^{-1})$; it gives the line-to-point alternative and the kernel--image relations stated in~(3).

We next compare these limits with Furstenberg convergence. Convergence in $\SP(3,\C)$ implies uniform convergence on compact subsets of the complement of the pseudo--projective kernel. Since the pointwise Furstenberg limits are taken along the same subsequence, they must agree there with the projectivizations of $\tau^+$ and $\tau^-$. Thus a rank-two limit with point kernel gives, on the complement of that point, precisely the rank-two map onto its image line; this is type~(2) in Definition~1.5. A rank-one limit with kernel line is constant on the complement of that line, with value its image point. The Furstenberg compactification may retain additional pointwise information on the exceptional line itself, but this does not alter the principal point--line datum. This proves the Furstenberg assertions. In particular, the rank of a pseudo--projective limit does not by itself determine the Furstenberg limit: the comparison identifies the common kernel--image and partial-flag data while allowing the pointwise compactification to retain additional information on the exceptional locus.

It remains to identify the wonderful boundary stratum. With the convention fixed above, the two normal parameters in the positive Weyl-chamber chart are
\[
t_1(g_n)=e^{-\delta_1(n)},
\qquad
t_2(g_n)=e^{-\delta_2(n)};
\]
see the description of the wonderful compactification recalled above and~\cite{DCP86}. Hence both gaps diverge exactly when both boundary parameters vanish, so that $\xi^+$ lies in
\[
D_{\alpha_1}\cap D_{\alpha_2},
\]
the closed orbit. If $\delta_1$ remains bounded and $\delta_2\to+\infty$, then only the $\alpha_2$ boundary parameter vanishes and $\xi^+\in D_{\alpha_2}^{\circ}$; if $\delta_1\to+\infty$ and $\delta_2$ remains bounded, then $\xi^+\in D_{\alpha_1}^{\circ}$. Finally, inversion interchanges the two simple gaps, since
\[
\delta_1(g_n^{-1})=\delta_2(g_n),
\qquad
\delta_2(g_n^{-1})=\delta_1(g_n).
\]
This gives the asserted strata for $\xi^-$.
\end{proof}

\subsection{Consequences for the Kulkarni strata}
We now pass from individual divergent sequences to the group as a whole. The next result identifies the first two Kulkarni layers in terms of images of Furstenberg limit maps.

\begin{proposition}
Let $G\subset\PSL(3,\C)$ be a discrete subgroup. Then
\[
L_0(G)\cup L_1(G)=\overline{\bigcup_{\rho\in\BFG}\Image(\rho)}.
\]
\end{proposition}
\begin{proof}
Set
\[
E_0(G):=\{x\in\Ptwo:\#\Stab_G(x)=\infty\}
\]
and let $E_1(G)$ denote the set of accumulation points of orbits $Gz$ with $z\in\Ptwo\setminus L_0(G)$. By definition of the first two Kulkarni strata,
\[
L_0(G)=\overline{E_0(G)},
\qquad
L_1(G)=\overline{E_1(G)}.
\]
We first show that
\[
E_0(G)\cup E_1(G)\subset
\bigcup_{\rho\in\BFG}\Image(\rho).
\]
If $x\in E_0(G)$, choose pairwise distinct elements $g_m\in G$ fixing $x$. After passing to a subsequence, the Furstenberg convergence property gives $g_m\to\rho\in\BFG$ pointwise on $\Ptwo$. Hence $\rho(x)=x$, and therefore $x\in\Image(\rho)$.

If $x\in E_1(G)$, there exist $z\notin L_0(G)$ and pairwise distinct orbit points $g_m(z)$ converging to $x$. Passing to a subsequence, we may assume that $g_m\to\rho\in\BFG$ pointwise. Thus $x=\rho(z)\in\Image(\rho)$. Taking closures, and using
\[
L_0(G)\cup L_1(G)
=
\overline{E_0(G)}\cup\overline{E_1(G)}
=
\overline{E_0(G)\cup E_1(G)},
\]
we obtain
\[
L_0(G)\cup L_1(G)\subset
\overline{\bigcup_{\rho\in\BFG}\Image(\rho)}.
\]

For the reverse inclusion, let $\rho\in\BFG$ and let $p\in\Image(\rho)$. Choose $z\in\Ptwo$ with $\rho(z)=p$ and a sequence of pairwise distinct elements $g_m\in G$ converging pointwise to $\rho$. Then $g_m(z)\to p$.

Suppose first that $z\notin L_0(G)$. The set $\{g_m(z):m\ge1\}$ must be infinite. Indeed, if it were finite, one of its values would occur for infinitely many indices. For infinitely many pairs $m\ne n$ we would then have $g_m(z)=g_n(z)$, hence $g_n^{-1}g_m\in\Stab_G(z)$; this would force $z\in E_0(G)\subset L_0(G)$, a contradiction. After passing to a subsequence, the points $g_m(z)$ are therefore pairwise distinct, so $p$ is an accumulation point of the orbit $Gz$. Hence $p\in E_1(G)\subset L_1(G)$.

If instead $z\in L_0(G)$, then $L_0(G)$ is $G$-invariant and closed. Consequently $g_m(z)\in L_0(G)$ for every $m$, and the convergence $g_m(z)\to p$ gives $p\in L_0(G)$. Thus
\[
\bigcup_{\rho\in\BFG}\Image(\rho)\subset L_0(G)\cup L_1(G).
\]
Since the right-hand side is closed, taking the closure on the left gives the opposite inclusion and proves the equality.
\end{proof}

\begin{remark}
If $G\subset\PSL(3,\C)$ is infinite and discrete, then $\Lambda_{\Kul}(G)$ contains at least one projective line; see~\cite{BCN16}.
\end{remark}

\section{Asymptotic flag data and limit sets}
We now prove Theorem~2. We first recall the projective boundary descriptions that will be used. Let $G\subset\PSL(3,\C)$ be discrete and write
\[
\BSPG=\{S\in\SP(3,\C):S\text{ is a limit of a sequence of distinct elements of }G\}.
\]
For a closed set $A\subset\Fl(\Ptwo)$, define its incidence Schubert thickening in $\Ptwo$ by
\begin{equation}
\Thinc(A):=\{x\in\Ptwo:\text{ there exists }(p,\ell)\in A\text{ with }x\in\ell\}
=\bigcup_{(p,\ell)\in A}\ell.
\end{equation}
For a single flag $(p,\ell)$, this is the Schubert cycle $\ell$; it is the type-$A_2$ incidence thickening appearing in regular flag dynamics; compare~\cite[Example~3.40]{KLP18}.

\subsection{Known projective boundary descriptions}
We shall use the following pseudo--projective description without reproving it. For every discrete subgroup $G\subset\PSL(3,\C)$,
\begin{equation}
\Eq(G)=\Ptwo\setminus\bigcup_{S\in\BSPG}\Ker(S),
\qquad
\Lambda_{\Myr}(G)=\bigcup_{S\in\BSPG}\Ker(S).
\end{equation}
See~\cite[Proposition~2.2]{CCL26} and the earlier pseudo--projective analysis in~\cite{BCN11}. In particular, no additional closure is required in~(2).

We use the terminology of Barrera--González-Urquiza--Navarrete~\cite{BGN18}. Thus $\widehat\Lambda_{\CoG}(G)\subset\Pdual$ denotes their extended Conze--Guivarc'h limit set: it is the set of limit points for the induced action of $G$ on the dual projective plane. Equivalently, $[\ell]\in\widehat\Lambda_{\CoG}(G)$ when there is a sequence of distinct elements of $G$ whose induced dual action converges to the constant point $[\ell]$ on a nonempty open subset of $\Pdual$. This is a dual-space construction; we do not use the same name for a separately defined set of point images in $\Ptwo$.

For comparison, define the pseudo--projective boundary sets
\[
A(G):=\{\Image(S):S\in\BSPG,\ \operatorname{rank}S=1\}\subset\Ptwo,
\]
\[
K(G):=\{[\Ker(S)]:S\in\BSPG,\ \operatorname{rank}S=1\}\subset\Pdual.
\]
Since $\SP(3,\C)$ is compact and $\BSPG$ is the accumulation set of $G$ in $\SP(3,\C)$, the set $\BSPG$ is closed. Its rank-one locus is therefore compact, and the image and kernel maps are continuous on that locus. Hence $A(G)$ and $K(G)$ are compact.

\begin{lemma}
Under the canonical bidual identification,
\[
K(G)=\widehat\Lambda_{\CoG}(G),
\qquad
A(G)=\widehat\Lambda_{\CoG}(\check G).
\]
In particular,
\[
\Lambda_{\CoG}(G)\subseteq\widehat\Lambda_{\CoG}(\check G),
\qquad
\Lambda_{\CoG}(\check G)\subseteq\widehat\Lambda_{\CoG}(G).
\]
\end{lemma}
\begin{proof}
Let $[\ell]\in K(G)$. Choose distinct $h_n\in G$ converging in $\SP(3,\C)$ to a rank-one map $S$ with $\Ker(S)=\ell$, and put $g_n=h_n^{-1}$. The induced action of $g_n$ on $\Pdual$ is represented by $h_n^T$. After choosing compatible scalar representatives, $h_n^T\to S^T$, whose image is the dual point $[\ell]$. Hence the dual action converges locally uniformly to $[\ell]$ off the kernel of $S^T$, so $[\ell]\in\widehat\Lambda_{\CoG}(G)$. Conversely, suppose a sequence in the dual action converges to a constant point $[\ell]$ on a nonempty open set. After passing to a subsequence with a pseudo--projective limit, that limit must have rank one: a rank-two projective map cannot be constant on a nonempty open set. Transposing and inverting as above produces a rank-one element $S\in\BSPG$ with $\Ker(S)=\ell$. Thus $K(G)=\widehat\Lambda_{\CoG}(G)$.

Apply the same statement to $\check G$. The dual of the contragredient action is canonically the original action on $\Ptwo$, so $K(\check G)=A(G)$. Hence $A(G)=\widehat\Lambda_{\CoG}(\check G)$. Finally, powers of a proximal element converge pseudo--projectively to its attracting rank-one projector, which gives $\Lambda_{\CoG}(G)\subseteq A(G)$. Applying the same argument to $\check G$ gives $\Lambda_{\CoG}(\check G)\subseteq K(G)$.
\end{proof}

\begin{remark}
No irreducibility or proximality assumption is required for the definition of the extended Conze--Guivarc'h set or for Lemma~3.1. Under the standard strong irreducibility and proximality hypotheses of Conze--Guivarc'h theory~\cite{CG00}, the contragredient action satisfies the same hypotheses. Together with the complex projective dual extension of~\cite{BGN18}, this yields equalities in the two inclusions of Lemma~3.1:
\[
\Lambda_{\CoG}(G)=\widehat\Lambda_{\CoG}(\check G),
\qquad
\Lambda_{\CoG}(\check G)=\widehat\Lambda_{\CoG}(G).
\]
Indeed, in this setting the classical Conze--Guivarc'h limit set is the unique minimal nonempty closed invariant set~\cite{CG00}, while rank-one boundary images arise as limits of attracting fixed points of proximal elements in the dual projective formulation of~\cite{BGN18}. These equalities are not used below.
\end{remark}

\subsection{The \texorpdfstring{$\theta$}{theta}-divergent regime}
Assume now that $G$ is $\theta$-divergent. Every $S\in\BSPG$ has rank one. Indeed, choose pairwise distinct $g_n\in G$ converging to $S$ in $\SP(3,\C)$ and write
\[
g_n=k_na_nr_n,
\qquad
a_n=\operatorname{diag}(e^{\lambda_1^{(n)}},e^{\lambda_2^{(n)}},e^{\lambda_3^{(n)}}).
\]
The two simple-root gaps tend to $+\infty$. After passing to a subsequence with $k_n\to k$ and $r_n\to r$, normalization by the largest singular value gives
\[
e^{-\lambda_1^{(n)}}a_n\longrightarrow\operatorname{diag}(1,0,0).
\]
Thus
\[
\Image(S)=\{k[e_1]\},
\qquad
\Ker(S)=r^{-1}\mathbb P\langle e_2,e_3\rangle.
\]
These are respectively the point component of the limiting flag of $(g_n)$ and the line component of the limiting flag of $(g_n^{-1})$.

Taking all such limits, and conversely all limiting flags, gives
\begin{equation}
A(G)=\pi_1(\Lambda_\theta(G)).
\end{equation}
For the line data, the same argument applied to inverse sequences gives
\begin{equation}
K(G)=\pi_2(\Lambda_\theta(G)).
\end{equation}
Equivalently, under $\wedge^2V\simeq V^*$ the two simple-root gaps are interchanged,
\[
\delta_1(\wedge^2g)=\delta_2(g),
\qquad
\delta_2(\wedge^2g)=\delta_1(g),
\]
which is the dual form of the same statement. Combining Lemma~3.1 with~(3)--(4) yields
\begin{equation}
\widehat\Lambda_{\CoG}(\check G)=\pi_1(\Lambda_\theta(G)),
\qquad
\widehat\Lambda_{\CoG}(G)=\pi_2(\Lambda_\theta(G)).
\end{equation}
Every pseudo--projective boundary map now has a line kernel. Equations~(2) and~(4) therefore give
\begin{equation}
\Lambda_{\Myr}(G)
=\bigcup_{S\in\BSPG}\Ker(S)
=\bigcup_{[\ell]\in\widehat\Lambda_{\CoG}(G)}\ell
=\bigcup_{(p,\ell)\in\Lambda_\theta(G)}\ell
=\Thinc(\Lambda_\theta(G)).
\end{equation}
Consequently
\begin{equation}
\Eq(G)=\Ptwo\setminus\Thinc(\Lambda_\theta(G)).
\end{equation}
We record the short source--target argument for proper discontinuity. If a compact set $K\subset\Eq(G)$ met infinitely many of its translates, choose pairwise distinct $g_n\in G$ and $x_n,y_n\in K$ with $y_n=g_nx_n$. After passing to a subsequence, $x_n\to x$, $y_n\to y$, and the full-flag case gives source and target flags $(p^+,\ell^+)$ and $(p^-,\ell^-)$ such that
\[
g_n\longrightarrow p^-
\quad\text{locally uniformly on }\Ptwo\setminus\ell^+.
\]
Since $\ell^+\subset\Thinc(\Lambda_\theta(G))$ and $K\subset\Eq(G)$, one has $x\notin\ell^+$, hence $y_n\to p^-$. But $p^-\in\ell^-\subset\Thinc(\Lambda_\theta(G))$, contradicting $y\in K\subset\Eq(G)$. Thus the action is properly discontinuous. This is the concrete type-$A_2$ instance of the regular Schubert-thickening mechanism of Kapovich--Leeb--Porti~\cite[Theorem~6.19]{KLP18}.

\subsection{Kulkarni consequence under three lines}
Assume in addition that $\Lambda_{\Kul}(G)$ contains at least three projective lines in general position. The equality of the equicontinuity and Kulkarni regions is already known in this regime:
\[
\Omega_{\Kul}(G)=\Eq(G),
\qquad
\Lambda_{\Kul}(G)=\Lambda_{\Myr}(G);
\]
see~\cite[Theorem~2.6]{BCN11}. Combining this with~(6) gives
\[
\Lambda_{\Kul}(G)=\Lambda_{\Myr}(G)
=\bigcup_{[\ell]\in\widehat\Lambda_{\CoG}(G)}\ell
=\Thinc(\Lambda_\theta(G)),
\]
\[
\Omega_{\Kul}(G)=\Eq(G)=\Ptwo\setminus\Thinc(\Lambda_\theta(G)).
\]
This is the last assertion of Theorem~2.

\subsection{Boundary parametrizations}
The compatibility established in the previous section shows that, for a $\theta$-divergent sequence, the pseudo--projective, pointwise Furstenberg-type, and wonderful compactifications encode the same asymptotic flag data. The present section identifies the projective limit-set constructions with the two natural projections of that flag boundary.

\begin{remark}
If in a more restrictive situation one has a parametrization
\[
\xi:X\longrightarrow\Fl(\Ptwo),
\qquad
\Lambda_\theta(G)=\xi(X),
\]
then, writing $\xi=(\xi^1,\xi^2)$, Theorem~2 gives formally
\[
\widehat\Lambda_{\CoG}(\check G)=\xi^1(X),
\qquad
\widehat\Lambda_{\CoG}(G)=\xi^2(X),
\]
and
\[
\Lambda_{\Myr}(G)=\bigcup_{x\in X}\xi^2(x),
\qquad
\Eq(G)=\Ptwo\setminus\bigcup_{x\in X}\xi^2(x).
\]
This is a specialization of Theorem~2 rather than an additional result, and it is not used elsewhere in the paper.
\end{remark}

\begin{corollary}
Let $G\subset\PSL(3,\C)$ be a complex Kleinian group and assume that $\Lambda_{\Kul}(G)$ contains at least three projective lines in general position. Then $\Omega_{\Kul}(G)=\Eq(G)$ and
\[
\overline{\bigcup_{\substack{g\in G\text{ loxodromic}\\ \ell_g^-\text{ is defined}}}\ell_g^-}
\subseteq\Lambda_{\Kul}(G)\subset\Ptwo,
\]
where $\ell_g^-$ denotes the repelling projective line when the repelling object of $g$ is a line, equivalently when $\wedge^2g^{-1}$ is proximal.
\end{corollary}
\begin{proof}
By~\cite[Theorem~2.6]{BCN11}, the three-line hypothesis gives
\[
\Omega_{\Kul}(G)=\Eq(G),
\qquad
\Lambda_{\Kul}(G)=\Lambda_{\Myr}(G).
\]
Let $g\in G$ be loxodromic with repelling line $\ell_g^-$. By the loxodromic classification recalled in Remark~1.3, the case in which the repelling object is a line is precisely the case in which the forward dynamics has an attracting point and $g$ is proximal. Hence, after normalizing lifts by the dominant eigenvalue, the powers $g^n$ converge pseudo--projectively to a rank-one map $S\in\BSPG$ whose exceptional hyperplane is exactly the repelling line:
\[
\Ker(S)=\ell_g^-.
\]
Equivalently, Lemma~3.1 identifies the dual point $[\ell_g^-]$ with an element of the extended Conze--Guivarc'h boundary coming from this rank-one kernel. Equation~(2) therefore gives
\[
\ell_g^-\subset\Lambda_{\Myr}(G).
\]
Using only the cited three-line equality,
\[
\ell_g^-\subset\Lambda_{\Myr}(G)=\Lambda_{\Kul}(G).
\]
Since $\Lambda_{\Kul}(G)$ is closed, taking the union over all such loxodromic elements and then the closure proves the claim. No part of Theorem~2 is used in this argument.
\end{proof}

\begin{remark}
The closure in the preceding statement is essential. Without an additional theorem identifying the ordinary and extended Conze--Guivarc'h limit sets, the argument above gives the inclusion needed here but does not assert that every line in $\Lambda_{\Kul}(G)$ is generated by the attracting direction of a proximal element of the dual action. No such equality is used below.
\end{remark}

This corollary provides a dynamical source of line configurations in $\Lambda_{\Kul}(G)$: repelling lines of loxodromic elements, together with their accumulation lines, lie in the Kulkarni limit set. These configurations will be used to study the Kobayashi geometry of the ordinary set. The corollary also yields $\Lambda_{\CoG}(G)\subset\widehat\Lambda_{\CoG}(\check G)$; these limit sets play a role for discrete subgroups of $\PSL(3,\C)$ analogous to that of the first and second Julia sets in higher-dimensional complex dynamics; compare~\cite{FS94,FS95}.

\section{Geometry of the ordinary set}\label{sec:geometry-ordinary}
The preceding description of the Kulkarni limit set reduces the geometry of the ordinary set to that of complements of projective line configurations. We use standard facts about the Kobayashi distance, including monotonicity, completeness, the induced topology, and hyperbolic embedding, as in~\cite{Kob98}. The analytic input is Green's theorem~\cite{Green77}.

For a discrete subgroup $G<\PSL(3,\C)$, let $\nu(G)$ denote the supremum of the cardinalities of families of projective lines contained in $\Lambda_{\Kul}(G)$ and in general position. Thus $\nu(G)=\infty$ means that arbitrarily large finite families occur.

We begin with the local criterion. It is useful independently of group actions.

\begin{theorem}\label{thm:local-five-line}
Let $\mathcal C$ be a family of projective lines in $\Ptwo$ such that $\bigcup_{\ell\in\mathcal C}\ell$ is closed, and set
\[
\Omega:=\Ptwo\setminus\bigcup_{\ell\in\mathcal C}\ell.
\]
Assume that $\mathcal C$ contains five lines in general position and that, for every $p\in\partial\Omega$, there exist five lines
\[
\ell_0,\ell_1,\ell_2,\ell_3,\ell_4\in\mathcal C
\]
in general position with $p\in\ell_0$. Then every connected component $\Omega_0$ of $\Omega$ is complete Kobayashi hyperbolic and Kobayashi hyperbolically embedded in $\Ptwo$.
\end{theorem}
\begin{proof}
Choose five lines of $\mathcal C$ in general position and let $D$ be their complement. By Green's theorem~\cite{Green77}, $D$ is complete Kobayashi hyperbolic and Kobayashi hyperbolically embedded in $\Ptwo$. Since $\Omega_0\subset D$, monotonicity gives
\[
k_{\Omega_0}\ge k_D,
\]
so $\Omega_0$ is Kobayashi hyperbolic.

Let $(z_n)$ be a $k_{\Omega_0}$-Cauchy sequence. By compactness of $\Ptwo$, a subsequence converges in the ambient topology to some $p\in\overline{\Omega_0}$. If $p\in\partial\Omega_0$, choose five lines $\ell_0,\ldots,\ell_4\in\mathcal C$ in general position with $p\in\ell_0$, and put
\[
D_p:=\Ptwo\setminus(\ell_0\cup\cdots\cup\ell_4).
\]
Again Green's theorem implies that $D_p$ is complete Kobayashi hyperbolic. Since $\Omega_0\subset D_p$ and $k_{D_p}\le k_{\Omega_0}$, the chosen subsequence is $k_{D_p}$-Cauchy and hence converges in $D_p$. The Kobayashi topology of the hyperbolic manifold $D_p$ agrees with its usual topology, contradicting its ambient convergence to $p\notin D_p$. Therefore $p\in\Omega_0$. Since $\Omega_0$ is Kobayashi hyperbolic, its Kobayashi topology is the usual topology; the Cauchy property then implies that the whole sequence converges to $p$. Thus $\Omega_0$ is complete.

Finally, let $p\ne q$ belong to $\overline{\Omega_0}$. If both are in $\Omega_0$, positivity and continuity of $k_{\Omega_0}$ give the required local separation. Otherwise, after interchanging $p$ and $q$ if necessary, assume $p\in\partial\Omega_0$ and choose $D_p$ as above. Green's hyperbolic-embedding conclusion gives neighborhoods $U_p,U_q\subset\Ptwo$ and $\varepsilon>0$ such that
\[
k_{D_p}(x,y)\ge\varepsilon
\]
for $x\in D_p\cap U_p$ and $y\in D_p\cap U_q$. Since $\Omega_0\subset D_p$, monotonicity yields the same lower bound for $k_{\Omega_0}$. Hence $\Omega_0$ is Kobayashi hyperbolically embedded in $\Ptwo$.
\end{proof}

The elementary incidence observation which was implicit in the earlier form of the argument is the following.

\begin{lemma}\label{lem:prescribed-line-combinatorial}
Let $\mathcal C$ be a family of projective lines in $\Ptwo$ with $\nu(\mathcal C)=\infty$. Then every prescribed line $\ell_0\in\mathcal C$ can be completed to arbitrarily large finite families of lines in $\mathcal C$ in general position.
\end{lemma}
\begin{proof}
Fix $N\ge1$ and choose $2N+1$ lines of $\mathcal C$ in general position. After discarding $\ell_0$ if it occurs among them, and then discarding further lines if necessary, we obtain a family $\mathcal A\subset\mathcal C\setminus\{\ell_0\}$ in general position with $|\mathcal A|=2N$. For each $p\in\ell_0$, at most two lines of $\mathcal A$ pass through $p$, because no three lines of $\mathcal A$ are concurrent. Partition $\mathcal A$ according to the intersection point with $\ell_0$ and retain one line from each nonempty class. The resulting subfamily $\mathcal A'\subset\mathcal A$ has at least $N$ elements, no two of which meet on $\ell_0$. Since $\mathcal A$ was already in general position, the family
\[
\{\ell_0\}\cup\mathcal A'
\]
is in general position. As $N$ is arbitrary, the conclusion follows.
\end{proof}

We can now recover the global four-line statement in the form used in the introduction.

\begin{theorem}\label{thm:global-four-lines}
Let $G<\PSL(3,\C)$ be a discrete subgroup. If $\nu(G)\ge4$, then every connected component $\Omega_0$ of $\Omega_{\Kul}(G)$ is complete Kobayashi hyperbolic and Kobayashi hyperbolically embedded in $\Ptwo$. Moreover, $\Omega_0$ is taut, pseudoconvex, Stein, and a domain of holomorphy.
\end{theorem}
\begin{proof}
There are two cases.

Assume first that $\nu(G)=4$. By the four-line classification~\cite{BCNfour11}, a finite-index subgroup is hyperbolic toral and the Kulkarni ordinary region is projectively equivalent to four disjoint copies of $\mathbb H\times\mathbb H$. Hence every component is complete Kobayashi hyperbolic. The product model also gives hyperbolic embedding. Indeed, after a projective change of coordinates, write a component as $\mathbb H\times\mathbb H\subset\C^2\subset\Ptwo$. If hyperbolic embedding failed, there would be sequences $z_n,w_n$ in the component converging in $\Ptwo$ to distinct points and satisfying $k_{\mathbb H\times\mathbb H}(z_n,w_n)\to0$. The product formula and the explicit Poincar\'e distance on $\mathbb H$ then give, for $j=1,2$,
\[
\frac{\operatorname{Im}w_{j,n}}{\operatorname{Im}z_{j,n}}\longrightarrow1,
\qquad
\frac{|z_{j,n}-w_{j,n}|}{\operatorname{Im}z_{j,n}}\longrightarrow0.
\]
For the homogeneous representatives $Z_n=(z_{1,n},z_{2,n},1)$ and $W_n=(w_{1,n},w_{2,n},1)$ it follows that
\[
\frac{\|Z_n-W_n\|}{\|Z_n\|}\longrightarrow0.
\]
Thus $z_n$ and $w_n$ have the same projective limit, a contradiction.

Assume now that $\nu(G)>4$. By~\cite[Theorem~1.2]{BCN16}, one has $\nu(G)=\infty$; moreover,~\cite[Proposition~5.15]{BCN16} shows that $\Lambda_{\Kul}(G)=\Ptwo\setminus\Omega_{\Kul}(G)$ is the union of a perfect set of projective lines. Hence, if $\mathcal C$ denotes the family of all projective lines contained in $\Lambda_{\Kul}(G)$, then
\[
\Lambda_{\Kul}(G)=\bigcup_{\ell\in\mathcal C}\ell.
\]
The union is closed because $\Lambda_{\Kul}(G)$ is closed. For each $p\in\Lambda_{\Kul}(G)$ choose a line $\ell_0\in\mathcal C$ through $p$. Lemma~\ref{lem:prescribed-line-combinatorial} completes this prescribed line to five lines of $\mathcal C$ in general position. Theorem~\ref{thm:local-five-line} then gives completeness and hyperbolic embedding.

The remaining analytic properties now follow in the usual way. Choose a projective line $\ell\subset\Lambda_{\Kul}(G)$. Then
\[
\Omega_0\subset\Ptwo\setminus\ell\simeq\C^2.
\]
A complete Kobayashi hyperbolic manifold is taut~\cite{Kob98}; a taut domain in $\C^2$ is pseudoconvex~\cite{Wu67}; and by the Levi problem a pseudoconvex domain in $\C^2$ is a domain of holomorphy, equivalently a Stein domain~\cite{Demailly12}. This proves the remaining assertions.
\end{proof}

\begin{corollary}\label{cor:quotient-complete}
Under the hypotheses of Theorem~\ref{thm:global-four-lines}, let $\Omega_0$ be a component of $\Omega_{\Kul}(G)$ and let $G_0=\operatorname{Stab}_G(\Omega_0)$. If $G_0$ acts freely on $\Omega_0$---in particular, if $G$ is torsion-free---then $\Omega_0/G_0$ is complete Kobayashi hyperbolic.
\end{corollary}
\begin{proof}
The projection $\Omega_0\to\Omega_0/G_0$ is a holomorphic covering. Complete Kobayashi hyperbolicity is invariant under holomorphic coverings; see~\cite{Kob98}.
\end{proof}

The two cases have different origins. When $\nu=4$ the geometry comes from the toral classification. When $\nu>4$, the local five-line mechanism explains why Green's theorem applies at every boundary point.

\section{An open Schottky locus}
We now prove Theorem~4 using the usual complex projective Schottky construction; see~\cite{SV03} and~\cite{CNS13}. We choose the initial configuration so that the required conditions are strict; they then persist under small perturbations.

Fix $r\ge2$ and identify
\[
X_r:=\Hom(F_r,\PSL(3,\C))\simeq\PSL(3,\C)^r.
\]
\begin{lemma}\label{lem:stable-regular-schottky}
For every $r\ge2$ there exists a tuple
\[
(g_1,\ldots,g_r)\in\PSL(3,\C)^r
\]
and a neighborhood $\mathcal U_r$ of this tuple, open in the usual complex topology, such that, for every $\rho=(h_1,\ldots,h_r)\in\mathcal U_r$:
\begin{enumerate}[label=(\roman*)]
\item the elements $h_i^{\pm1}$ are proximal and satisfy a projective ping--pong system; hence $\rho$ is injective and $\Gamma_\rho:=\langle h_1,\ldots,h_r\rangle$ is free and discrete;
\item $\Gamma_\rho$ has nonempty equicontinuity region, hence nonempty Kulkarni ordinary set;
\item the first two generators have no common fixed point and no common invariant projective line;
\item four distinguished invariant lines of $h_1^{\pm1},h_2^{\pm1}$ are in general position;
\item $\Gamma_\rho$ is $\theta$-divergent.
\end{enumerate}
\end{lemma}
\begin{proof}
Choose, for each $i$, an ordered projective eigenframe
\[
([v_{i,1}],[v_{i,2}],[v_{i,3}])
\]
in generic relative position. The eigenframes may be chosen so that the attracting points and repelling lines of the $2r$ letters $g_i^{\pm1}$ satisfy the usual finite transversality conditions for projective ping--pong, while the four distinguished invariant lines of the first two generators are in general position. We also require that $g_1$ and $g_2$ have neither a common eigenpoint nor a common invariant projective line. We choose the eigenframes slightly more carefully so that the attracting and repelling full flags of the $2r$ letters $g_i^{\pm1}$ are pairwise opposite. These are finitely many nonincidence conditions and can be satisfied simultaneously by a generic choice of eigenframes.

In these frames take
\[
g_i=q_i\operatorname{diag}(R_i^2,1,R_i^{-2})q_i^{-1},
\qquad R_i>1.
\]
For $R_i$ sufficiently large, the proximal convergence to the attracting point is uniform on compacta away from the repelling line. Choosing small pairwise disjoint compact neighborhoods of the attracting points, all in a common affine chart and disjoint from the required repelling lines, gives the standard strict ping--pong inclusions. A base point outside these compact sets and repelling lines then determines a Schottky-type system. The usual ping--pong argument shows that the generated group is free and discrete. The same strict inclusions also give a neighborhood of the base point on which all but finitely many reduced words form a normal family; hence $\Eq(\Gamma_\rho)\ne\varnothing$ and therefore $\Omega_{\Kul}(\Gamma_\rho)\ne\varnothing$.

We also need full-chamber regularity. The diagonal elements above translate in a direction contained in the interior of the full Weyl chamber, and the chosen attracting and repelling full flags are pairwise opposite. The higher-rank Schottky construction of Kapovich--Leeb--Porti applies: after increasing the $R_i$ if necessary, the associated free action on the symmetric space is full-chamber Morse~\cite[Section~4.2, especially Theorem~4.13 and Remark~4.16]{KLP25}. In particular it is uniformly full-chamber regular~\cite[Remark~4.2]{KLP25}. Since the orbit map is a quasi-isometric embedding, every sequence of distinct elements has Cartan projection escaping to infinity while remaining uniformly away from both walls. Equivalently, both simple Cartan gaps tend to infinity, so $\Gamma_\rho$ is $\theta$-divergent.

All inequalities used here are strict. On the simple-spectrum locus the eigenpoints and invariant lines vary continuously. The Morse property is also open in the representation space~\cite[Theorem~1.2]{KLP25}. Thus, after shrinking the neighborhood of the initial tuple, the ping--pong inclusions, the four-line general-position condition, the absence of a common fixed point or invariant line for the first two generators, and $\theta$-divergence all persist. This gives the open set $\mathcal U_r$.
\end{proof}

\begin{lemma}\label{lem:strong-irr-open}
Every group $\Gamma_\rho$ with $\rho\in\mathcal U_r$ is strongly irreducible.
\end{lemma}
\begin{proof}
It is enough to consider $\Gamma_{12}:=\langle h_1,h_2\rangle$. Suppose that $\Gamma_{12}$ has a finite orbit in $\Ptwo$. Every point of that orbit is fixed by a positive power of each generator. Since $h_1$ and $h_2$ have three eigenvalues of pairwise distinct moduli, every positive power has the same eigendirections as the corresponding generator. Hence such a point would be fixed by both $h_1$ and $h_2$, contrary to Lemma~\ref{lem:stable-regular-schottky}(iii). Applying the same argument in the dual projective plane excludes a finite orbit of projective lines. Thus $\Gamma_{12}$, and hence $\Gamma_\rho$, is strongly irreducible.
\end{proof}

\begin{proposition}\label{prop:strong-irr-four-to-infty}
Let $\Gamma<\PSL(3,\C)$ be a discrete strongly irreducible group. If $\nu(\Gamma)\ge4$, then $\nu(\Gamma)=\infty$.
\end{proposition}
\begin{proof}
The line-counting theorem~\cite{BCN16} gives
\[
\nu(\Gamma)\in\{1,2,3,4,\infty\}.
\]
If $\nu(\Gamma)=4$, the four-line classification~\cite{BCNfour11} gives a finite-index subgroup preserving a proper projective subspace, contrary to strong irreducibility. Hence $\nu(\Gamma)=\infty$.
\end{proof}

\begin{proposition}\label{prop:nu-infty-open}
For every $\rho\in\mathcal U_r$ one has
\[
\nu(\Gamma_\rho)=\infty.
\]
\end{proposition}
\begin{proof}
By the cyclic projective classification~\cite[Proposition~4.2.28]{CNS13}, each of the four distinguished invariant lines supplied by Lemma~\ref{lem:stable-regular-schottky}(iv) belongs to the Kulkarni limit set of the corresponding cyclic subgroup. Since $\Gamma_\rho$ is strongly irreducible, it has neither a global fixed point nor a global invariant projective line; hence the non-elementary alternative of~\cite[Theorem~6.3.6(b)]{CNS13} applies and these cyclic limit sets are contained in $\Lambda_{\Kul}(\Gamma_\rho)$. Thus $\nu(\Gamma_\rho)\ge4$, and Proposition~\ref{prop:strong-irr-four-to-infty} gives the result.
\end{proof}

\begin{corollary}\label{cor:local-five-line-open-locus}
For every $\rho\in\mathcal U_r$ and every $p\in\Lambda_{\Kul}(\Gamma_\rho)$, there exist five lines
\[
\ell_0,\ldots,\ell_4\subset\Lambda_{\Kul}(\Gamma_\rho)
\]
in general position with $p\in\ell_0$. Consequently every connected component of $\Omega_{\Kul}(\Gamma_\rho)$ is complete Kobayashi hyperbolic and Kobayashi hyperbolically embedded in $\Ptwo$; it is also taut, pseudoconvex, Stein, and a domain of holomorphy.
\end{corollary}

\begin{proof}
Fix $\rho\in\mathcal U_r$ and put $\Gamma=\Gamma_\rho$. By Proposition~\ref{prop:nu-infty-open}, $\nu(\Gamma)=\infty$. In particular, $\Lambda_{\Kul}(\Gamma)$ contains three projective lines in general position, so the general-position hypothesis in Theorem~2 is satisfied. Since $\Gamma$ is also $\theta$-divergent, Theorem~2 gives
\[
\Lambda_{\Kul}(\Gamma)=\bigcup_{[\ell]\in\widehat\Lambda_{\CoG}(\Gamma)}\ell.
\]
Thus every $p\in\Lambda_{\Kul}(\Gamma)$ lies on a line $\ell_0\subset\Lambda_{\Kul}(\Gamma)$. Lemma~\ref{lem:prescribed-line-combinatorial} completes this prescribed line to five limit lines in general position. Theorem~\ref{thm:local-five-line} gives completeness and hyperbolic embedding, and the last assertions follow as in Theorem~\ref{thm:global-four-lines}.
\end{proof}

\begin{proof}[Proof of Theorem 4]
By Lemmas~\ref{lem:stable-regular-schottky} and~\ref{lem:strong-irr-open}, Propositions~\ref{prop:strong-irr-four-to-infty} and~\ref{prop:nu-infty-open}, and Corollary~\ref{cor:local-five-line-open-locus}, the set $\mathcal U_r$ is a nonempty subset of $X_r$, open in the usual complex topology, and every $\rho\in\mathcal U_r$ has all the properties stated in the theorem.

The complex algebraic variety
\[
X_r\simeq\PSL(3,\C)^r
\]
is irreducible. A proper Zariski-closed subset of an irreducible complex algebraic variety has empty interior in the usual complex topology. Hence every nonempty subset of $X_r$ that is open in the usual complex topology is Zariski dense, so
\[
\overline{\mathcal U_r}^{\,Z}=X_r.
\]
Since $\mathcal U_r\subset\mathcal K_r$, the locus with $\nu=\infty$ and the locus with the complete Kobayashi conclusion both contain $\mathcal U_r$. Their Zariski closures therefore contain $\mathcal K_r$.
\end{proof}

\begin{remark}
The numerical equality $\nu=\infty$ becomes geometric once one knows that every boundary point lies on a limit line. On the open locus this is supplied directly by $\theta$-divergence and Theorem~2; Lemma~\ref{lem:prescribed-line-combinatorial} then gives the local five-line condition. This is the same local mechanism which appears in the proof of Theorem~\ref{thm:global-four-lines}.
\end{remark}

\begin{corollary}\label{cor:strong-irr-zariski-dense}
The locus
\[
\mathcal I_r:=\{\rho\in\mathcal K_r:\rho(F_r)\text{ is strongly irreducible}\}
\]
is Zariski dense in $\mathcal K_r$.
\end{corollary}
\begin{proof}
Theorem~4 gives $\mathcal U_r\subset\mathcal I_r$ and $\overline{\mathcal U_r}^{\,Z}=X_r$. Hence the Zariski closure of $\mathcal I_r$ contains $\mathcal K_r$.
\end{proof}

\subsection{The unmarked Zariski--Chabauty topology}
We now pass from marked representations to their image subgroups. We shall use a final topology induced by the marked representation spaces. It is convenient first to compare it with Chabauty's classical construction on closed subgroups~\cite{Chab50}. The two topologies are not identified.

If $\PSL(3,\C)$ is endowed with its Zariski topology and $\operatorname{Sub}(\PSL(3,\C))$ denotes the set of all subgroups, one may also form a hit-and-miss topology generated by
\[
U^+(V):=\{H:H\cap V\ne\varnothing\}
\]
for Zariski-open $V\subset\PSL(3,\C)$ and
\[
U^-(Z):=\{H:H\cap Z=\varnothing\}
\]
for Zariski-closed $Z\subset\PSL(3,\C)$. We use this only for comparison.

For this comparison one may use the following sequential notation. If $(H_n)$ is a sequence of subgroups of $\PSL(3,\C)$, set
\[
\liminf_Z H_n=
\left\{g:\begin{array}{l}
\text{every Zariski-open neighborhood $V$ of $g$ meets $H_n$}\\
\text{for all sufficiently large $n$}
\end{array}\right\}
\]
and
\[
\limsup_Z H_n=
\left\{g:\begin{array}{l}
\text{every Zariski-open neighborhood $V$ of $g$ meets $H_n$}\\
\text{for infinitely many $n$}.
\end{array}\right\}.
\]
We shall say, only in this comparison paragraph, that $H_n$ converges Zariski-geometrically to $H_\infty$ if
\[
\liminf_ZH_n=\limsup_ZH_n=H_\infty.
\]

\begin{lemma}
Let $H_n\subset\PSL(3,\C)$ and put $A_n:=\overline{H_n}^{\,Z}$. Then
\[
\liminf_ZH_n=\liminf_ZA_n,
\qquad
\limsup_ZH_n=\limsup_ZA_n.
\]
Thus this hit-and-miss convergence depends only on the Zariski closures of the subgroups.
\end{lemma}

\begin{proof}
If $V\subset\PSL(3,\C)$ is Zariski open, then
\[
H_n\cap V\ne\varnothing
\quad\Longleftrightarrow\quad
A_n\cap V\ne\varnothing.
\]
Indeed, $H_n$ is Zariski dense in $A_n$, so every nonempty Zariski-open subset of $A_n$ meets $H_n$. The two equalities follow directly from the definitions.
\end{proof}

\begin{definition}
Let
\[
\mathcal D_{\mathrm{fg}}:=\bigcup_{r\ge2}\pi_r(\mathcal K_r),
\qquad
\pi_r:\mathcal K_r\longrightarrow\operatorname{Sub}(\PSL(3,\C)),
\qquad
\pi_r(\rho)=\rho(F_r).
\]
The Zariski--Chabauty topology used in this paper is the final topology on $\mathcal D_{\mathrm{fg}}$ induced by all the maps $\pi_r$. Thus $U\subset\mathcal D_{\mathrm{fg}}$ is open if and only if
\[
\pi_r^{-1}(U)
\]
is Zariski open in $\mathcal K_r$ for every $r\ge2$.
\end{definition}

Thus all finite markings enter the definition. In particular, if $G_0\in\mathcal D_{\mathrm{fg}}$, if $U$ is a Zariski--Chabauty neighborhood of $G_0$, and if $\rho_0\in\mathcal K_r$ is any marking with $\rho_0(F_r)=G_0$, then $\pi_r^{-1}(U)$ is a Zariski-open neighborhood of $\rho_0$ in $\mathcal K_r$.

For $d\ge2$, let $\mathcal D_d\subset\mathcal D_{\mathrm{fg}}$ be the set of groups that admit a marking by $F_r$ for some $2\le r\le d$, and give $\mathcal D_d$ the subspace topology. Thus $d$ bounds the number of generators of the group; the topology itself still remembers all finite markings through $\mathcal D_{\mathrm{fg}}$.

\begin{theorem}
For every $d\ge2$, the set
\[
\mathcal E_d:=\{G\in\mathcal D_d:\nu(G)=\infty\}
\]
is dense in $\mathcal D_d$ for the Zariski--Chabauty topology. More precisely, every Zariski--Chabauty neighborhood of every $G_0\in\mathcal D_d$ contains a strongly irreducible, $\theta$-divergent Schottky-type complex Kleinian group $G$ with $\nu(G)=\infty$ such that every point of $\Lambda_{\Kul}(G)$ lies on a limit line which can be completed to five limit lines in general position. Consequently every connected component of $\Omega_{\Kul}(G)$ is complete Kobayashi hyperbolic and Kobayashi hyperbolically embedded in $\Ptwo$, and hence is taut, pseudoconvex, Stein, and a domain of holomorphy.
\end{theorem}

\begin{proof}
Let $G_0\in\mathcal D_d$ and let $U\subset\mathcal D_d$ be an open neighborhood of $G_0$. By the subspace topology, there is an open set $O\subset\mathcal D_{\mathrm{fg}}$ such that
\[
U=O\cap\mathcal D_d.
\]
Choose a marking
\[
\rho_0:F_r\longrightarrow\PSL(3,\C),
\qquad
\rho_0(F_r)=G_0,
\qquad
2\le r\le d.
\]
By the definition of the final topology,
\[
\pi_r^{-1}(O)
\]
is a Zariski-open neighborhood of $\rho_0$ in $\mathcal K_r$. Hence there exists a Zariski-open subset $W$ of $X_r=\Hom(F_r,\PSL(3,\C))$ such that
\[
\rho_0\in W,
\qquad
W\cap\mathcal K_r\subset\pi_r^{-1}(O).
\]
By Theorem~4, the Schottky locus $\mathcal U_r$ is Zariski dense in $X_r$. Therefore
\[
W\cap\mathcal U_r\ne\varnothing.
\]
Choose $\rho$ in this intersection and put $G=\rho(F_r)$. Since $\mathcal U_r\subset\mathcal K_r$, we have $\rho\in\pi_r^{-1}(O)$, and therefore $G\in O$. The group $G$ is generated by at most $r\le d$ elements, so $G\in\mathcal D_d$; hence $G\in U$. Finally Theorem~4 gives
\[
\nu(G)=\infty,
\]
and the same group is strongly irreducible and $\theta$-divergent, satisfies the local five-line condition at every point of its Kulkarni limit set, and has the asserted complete Kobayashi geometry. Thus every neighborhood of $G_0$ contains such a group, proving the result.
\end{proof}

Here ``approximation'' means only membership in the closure: every open neighborhood contains one of the indicated groups. We do not assert that the final Zariski--Chabauty topology is sequential.

\section{The Zassenhaus obstruction}\label{sec:zassenhaus-obstruction}
The preceding density statements use two different topologies: the Zariski topology on the marked representation spaces and the final Zariski--Chabauty topology on the image groups. Neither statement gives density for the usual topology of a representation variety, and the unmarked statement does not give density for the usual convergence topologies on closed subgroups. We record a simple obstruction.

For a fixed finitely generated group $\Gamma$, a sequence of representations
\[
\rho_j:\Gamma\longrightarrow\PSL(3,\C)
\]
is said to converge \emph{algebraically} to $\rho$ if
\[
\rho_j(\gamma)\longrightarrow\rho(\gamma)
\qquad\text{for every }\gamma\in\Gamma.
\]
For $\Gamma=F_r$, this is exactly convergence in the usual complex topology of $X_r$, since it is enough to check convergence on a free basis.

We also use the classical Chabauty topology on closed subgroups of the Lie group $\PSL(3,\C)$~\cite{Chab50}. A sequence of closed subgroups $G_j$ converges \emph{geometrically} to $G$ when it converges to $G$ in this topology. We shall use the following form of strong convergence, standard in the Kleinian-group literature~\cite{McMullen99,KapovichSequences}: $G_j$ converges \emph{strongly} to a finitely generated group $G$ if $G_j\to G$ geometrically and, for all sufficiently large $j$, there are surjective homomorphisms
\[
\chi_j:G\longrightarrow G_j
\]
such that $\chi_j(g)\to g$ for every $g\in G$. This is the only form of strong convergence needed below.

We use the following Zassenhaus-neighborhood form of the Kazhdan--Margulis theorem.

\begin{lemma}\label{lem:zassenhaus-neighborhood-current}
There exists a neighborhood $U$ of the identity in $\PSL(3,\C)$ such that, for every discrete subgroup $\Delta\subset\PSL(3,\C)$, the subgroup generated by $\Delta\cap U$ is nilpotent.
\end{lemma}

\begin{proof}
This is the standard Zassenhaus-neighborhood theorem for semisimple Lie groups; see, for example,~\cite[Chapter~4]{KapovichBook}.
\end{proof}

The next example shows that the Zariski-density statement does not extend to the usual topology.

\begin{proposition}\label{prop:zassenhaus-usual-open-obstruction}
There exist a representation $\rho_0\in\mathcal K_2$ and a neighborhood $\mathcal V$ of $\rho_0$ in the usual topology of $X_2$ such that, for every
\[
\rho\in\mathcal V\cap\mathcal K_2,
\]
the group $\rho(F_2)$ is nilpotent. In particular,
\[
\mathcal V\cap\mathcal U_2=\varnothing.
\]
Consequently $\mathcal U_2$, although Zariski dense, is not dense in $\mathcal K_2$ for the usual topology.
\end{proposition}

\begin{proof}
Let $U$ be a Zassenhaus neighborhood as in Lemma~\ref{lem:zassenhaus-neighborhood-current}. For $t>0$ set
\[
a_t=
\left[
\begin{array}{ccc}
 e^t&0&0\\
 0&e^{-t}&0\\
 0&0&1
\end{array}
\right],
\qquad
b_t=
\left[
\begin{array}{ccc}
 1&0&0\\
 0&e^t&0\\
 0&0&e^{-t}
\end{array}
\right].
\]
Choose $t>0$ sufficiently small that $a_t,b_t\in U$, and define
\[
\rho_0:F_2=\langle x_1,x_2\rangle\longrightarrow\PSL(3,\C),
\qquad
\rho_0(x_1)=a_t,
\quad
\rho_0(x_2)=b_t.
\]
The two elements commute, and
\[
a_t^m b_t^n=
\left[
\begin{array}{ccc}
 e^{mt}&0&0\\
 0&e^{(-m+n)t}&0\\
 0&0&e^{-nt}
\end{array}
\right].
\]
On the invariant open set
\[
\mathcal O:=\{[z_1:z_2:z_3]\in\Ptwo:z_1z_2z_3\ne0\}\simeq(\C^*)^2,
\]
use the coordinates $u=z_1/z_3$ and $v=z_2/z_3$. Then
\[
a_t^m b_t^n(u,v)=
\bigl(e^{(m+n)t}u,e^{(-m+2n)t}v\bigr).
\]
On logarithmic moduli this is translation by
\[
t\,(m+n,-m+2n).
\]
The homomorphism
\[
\mathbb Z^2\longrightarrow\mathbb R^2,
\qquad
(m,n)\longmapsto t\,(m+n,-m+2n)
\]
has lattice image, since the defining linear map has determinant $3$. Hence
\[
G_0:=\rho_0(F_2)=\langle a_t,b_t\rangle\simeq\mathbb Z^2
\]
is discrete and acts properly discontinuously on $\mathcal O$. The same lattice-translation description also shows that no point of $\mathcal O$ belongs to any of the three Kulkarni sets $L_0,L_1,L_2$: stabilizers are trivial, orbits have no accumulation point in $\mathcal O$, and translates of compact subsets are locally finite there. Thus
\[
\mathcal O\subset\Omega_{\Kul}(G_0),
\]
so $G_0$ is a non-virtually-cyclic complex Kleinian group and $\rho_0\in\mathcal K_2$.

Let
\[
\operatorname{ev}:X_2\longrightarrow\PSL(3,\C)^2
\]
be evaluation on $(x_1,x_2)$ and put
\[
\mathcal V:=\operatorname{ev}^{-1}(U\times U).
\]
This is an open neighborhood of $\rho_0$ in the usual topology. If $\rho\in\mathcal V\cap\mathcal K_2$ and $G=\rho(F_2)$, then $G$ is discrete and both marked generators lie in $G\cap U$. Therefore
\[
G=\langle\rho(x_1),\rho(x_2)\rangle
\subseteq\langle G\cap U\rangle
\subseteq G,
\]
so
\[
G=\langle G\cap U\rangle.
\]
Lemma~\ref{lem:zassenhaus-neighborhood-current} implies that $G$ is nilpotent. A finitely generated nilpotent complex linear group is virtually triangularizable; equivalently, a finite-index subgroup preserves a complete projective flag~\cite{Malcev56}. Such a group is not strongly irreducible. Since every group in $\mathcal U_2$ is strongly irreducible, we obtain
\[
\mathcal V\cap\mathcal U_2=\varnothing.
\]
\end{proof}

\begin{corollary}\label{cor:no-algebraic-density-current}
The representation $\rho_0$ of Proposition~\ref{prop:zassenhaus-usual-open-obstruction} is not an algebraic limit of representations in $\mathcal U_2$. Thus the Zariski-density assertion of Theorem~4 does not extend to density for algebraic convergence.
\end{corollary}

\begin{proof}
For $F_2$, algebraic convergence is the usual topology on $X_2$. The neighborhood $\mathcal V$ in Proposition~\ref{prop:zassenhaus-usual-open-obstruction} is open in the usual topology and is disjoint from $\mathcal U_2$.
\end{proof}

The same example gives an obstruction after forgetting the marking.

\begin{corollary}\label{cor:no-strong-density-current}
Let $G_0=\rho_0(F_2)\simeq\mathbb Z^2$. Then $G_0$ is not a strong limit of groups of the form $\rho(F_2)$ with $\rho\in\mathcal U_2$. In particular, the unmarked Zariski--Chabauty density theorem does not imply density for strong convergence.
\end{corollary}

\begin{proof}
Suppose that $G_j=\rho_j(F_2)$ with $\rho_j\in\mathcal U_2$ and that $G_j$ converges strongly to $G_0$. For all sufficiently large $j$, strong convergence gives a surjective homomorphism
\[
\chi_j:G_0\longrightarrow G_j.
\]
Since $G_0\simeq\mathbb Z^2$ is abelian, every quotient $G_j$ is abelian. On the other hand, $\rho_j$ is injective by Theorem~4, so $G_j\simeq F_2$ is nonabelian. This is impossible.
\end{proof}

Thus the two density statements are specific to the topologies in which they are formulated. The marked result is a Zariski-density statement, while the unmarked result uses the final Zariski--Chabauty topology. We make no claim about density for geometric convergence alone; in particular, the classical Chabauty topology remains separate from the final topology introduced above.

\section*{Funding}
The research of A. Cano was partially supported by SECIHTI-SNII 104023 and PAPIIT-UNAM IN112424.

\section*{Declarations}
All authors contributed equally to this work and declare no conflicts of interest.

\medskip
\noindent\textbf{Use of AI-assisted tools.} During the preparation of this manuscript, the authors used ChatGPT (OpenAI) for editorial assistance, structural review, and the examination of intermediate arguments and references. All mathematical statements, proofs, citations, and the final text were independently reviewed and approved by the authors, who take full responsibility for the content.

\section*{Acknowledgments}
The authors are indebted to Dennis Sullivan, Bill Goldman, Misha Kapovich, Alberto Verjovsky, John Parker, and Dick Canary for illuminating conversations, and to Mauricio Toledo-Acosta for helpful discussions. They also thank the Instituto de Matemáticas, UNAM, Unidad Cuernavaca, and the Facultad de Matemáticas, UADY, and their members for their hospitality and support during the preparation of this paper.

\bigskip
\noindent UADY, Anillo Periférico Norte, Tablaje Cat. 13615, Colonia Chuburná Hidalgo Inn, Mérida Yucatán, México.\\
\textit{Email address:} \texttt{bvargas@uady.mx}

\medskip
\noindent IMUNAM, Unidad Cuernavaca, Av. Universidad s/n. Col. Lomas de Chamilpa, C.P. 62210, Cuernavaca, Morelos, México.\\
\textit{Email address:} \texttt{angelcano@im.unam.mx}

\medskip
\noindent UADY, Anillo Periférico Norte, Tablaje Cat. 13615, Colonia Chuburná Hidalgo Inn, Mérida Yucatán, México.\\
\textit{Email address:} \texttt{jp.navarrete@uady.mx}

\medskip
\noindent IMUNAM, Unidad Cuernavaca, Av. Universidad s/n. Col. Lomas de Chamilpa, C.P. 62210, Cuernavaca, Morelos, México.\\
\textit{Email address:} \texttt{jseade@im.unam.mx}


\begin{thebibliography}{99}
\bibitem{Benoist97} Y. Benoist, Propriétés asymptotiques des groupes linéaires, \emph{Geom. Funct. Anal.} \textbf{7} (1997), no.~1, 1--47. \href{https://doi.org/10.1007/PL00001613}{doi:10.1007/PL00001613}.

\bibitem{BCN11} W. Barrera Vargas, Á. Cano, and J. P. Navarrete Carrillo, The limit set of discrete subgroups of $\PSL(3,\C)$, \emph{Math. Proc. Cambridge Philos. Soc.} \textbf{150} (2011), no.~1, 129--146. \href{https://doi.org/10.1017/S0305004110000423}{doi:10.1017/S0305004110000423}.

\bibitem{BCNfour11} W. Barrera, Á. Cano, and J. P. Navarrete, Subgroups of $\PSL(3,\C)$ with four lines in general position in its limit set, \emph{Conform. Geom. Dyn.} \textbf{15} (2011), 160--176. \href{https://doi.org/10.1090/S1088-4173-2011-00231-1}{doi:10.1090/S1088-4173-2011-00231-1}.

\bibitem{BCN16} W. Barrera, Á. Cano, and J. P. Navarrete, On the number of lines in the limit set for discrete subgroups of $\PSL(3,\C)$, \emph{Pacific J. Math.} \textbf{281} (2016), no.~1, 17--49. \href{https://doi.org/10.2140/pjm.2016.281.17}{doi:10.2140/pjm.2016.281.17}.

\bibitem{BGN18} W. Barrera, A. González-Urquiza, and J. P. Navarrete, Duality of the Kulkarni limit set for subgroups of $\PSL(3,\C)$, \emph{Bull. Braz. Math. Soc. (N.S.)} \textbf{49} (2018), 261--277. \href{https://doi.org/10.1007/s00574-017-0053-9}{doi:10.1007/s00574-017-0053-9}.

\bibitem{BCNS25} W. Barrera, Á. Cano, J. P. Navarrete, and J. Seade, Elementary groups in $\PSL(3,\C)$, in \emph{Geometry, Groups and Mathematical Philosophy}, Contemp. Math., vol.~811, American Mathematical Society, Providence, RI, 2025, pp.~31--47. \href{https://doi.org/10.1090/conm/811/16235}{doi:10.1090/conm/811/16235}.

\bibitem{CCL26} Á. Cano, L. A. Cano Cordero, and L. Loeza, Regular sets for discrete projective group actions on $\mathbb{CP}^2$, \emph{Bull. Braz. Math. Soc. (N.S.)} \textbf{57} (2026), Art.~28. \href{https://doi.org/10.1007/s00574-026-00516-4}{doi:10.1007/s00574-026-00516-4}.

\bibitem{CNS13} Á. Cano, J. P. Navarrete, and J. Seade, \emph{Complex Kleinian Groups}, Progr. Math., vol.~303, Birkhäuser/Springer, Basel, 2013. \href{https://doi.org/10.1007/978-3-0348-0481-3}{doi:10.1007/978-3-0348-0481-3}.

\bibitem{CS10} Á. Cano and J. Seade, On the equicontinuity region of discrete subgroups of $\mathrm{PU}(1,n)$, \emph{J. Geom. Anal.} \textbf{20} (2010), no.~2, 291--305. \href{https://doi.org/10.1007/s12220-009-9107-6}{doi:10.1007/s12220-009-9107-6}.

\bibitem{Chab50} C. Chabauty, Limite d'ensembles et géométrie des nombres, \emph{Bull. Soc. Math. France} \textbf{78} (1950), 143--151. \href{https://doi.org/10.24033/bsmf.1412}{doi:10.24033/bsmf.1412}.


\bibitem{CG00} J.-P. Conze and Y. Guivarc'h, Limit sets of groups of linear transformations, \emph{Sankhyā Ser. A} \textbf{62} (2000), no.~3, 367--385. \href{https://doi.org/10.2307/25051328}{doi:10.2307/25051328}.

\bibitem{DCP86} C. De Concini and C. Procesi, Cohomology of compactifications of algebraic groups, \emph{Duke Math. J.} \textbf{53} (1986), no.~3, 585--594. \href{https://doi.org/10.1215/S0012-7094-86-05333-0}{doi:10.1215/S0012-7094-86-05333-0}.

\bibitem{Vain84} I. Vainsencher, Complete collineations and blowing up determinantal ideals, \emph{Math. Ann.} \textbf{267} (1984), 417--432. \href{https://doi.org/10.1007/BF01456098}{doi:10.1007/BF01456098}.

\bibitem{Demailly12} J.-P. Demailly, \emph{Complex Analytic and Differential Geometry}, OpenContent Book, Universit\'e Grenoble Alpes, Grenoble, 2012. \href{https://www-fourier.univ-grenoble-alpes.fr/~demailly/manuscripts/agbook.pdf}{Official manuscript}.

\bibitem{FS94} J. E. Fornæss and N. Sibony, Complex dynamics in higher dimension. I, \emph{Astérisque} \textbf{222} (1994), 201--231. \href{https://doi.org/10.24033/ast.259}{doi:10.24033/ast.259}.

\bibitem{FS95} J. E. Fornæss and N. Sibony, Complex dynamics in higher dimension. II, in \emph{Modern Methods in Complex Analysis}, T. Bloom, D. W. Catlin, J. P. D'Angelo, and Y.-T. Siu, eds., Ann. of Math. Stud., vol.~137, Princeton Univ. Press, Princeton, NJ, 1995, pp.~135--182.

\bibitem{Green77} M. L. Green, The hyperbolicity of the complement of $2n+1$ hyperplanes in general position in $\mathbb P^n$, and related results, \emph{Proc. Amer. Math. Soc.} \textbf{66} (1977), no.~1, 109--113. \href{https://doi.org/10.1090/S0002-9939-1977-0457790-9}{doi:10.1090/S0002-9939-1977-0457790-9}.

\bibitem{KapovichBook} M. Kapovich, \emph{Hyperbolic Manifolds and Discrete Groups}, Progr. Math., vol.~183, Birkh\"auser Boston, Boston, MA, 2001.

\bibitem{KapovichSequences} M. Kapovich, On sequences of finitely generated discrete groups, in \emph{In the Tradition of Ahlfors--Bers, V}, Contemp. Math., vol.~510, Amer. Math. Soc., Providence, RI, 2010, pp.~165--184. \href{https://doi.org/10.1090/conm/510/10024}{doi:10.1090/conm/510/10024}.




\bibitem{KLP18} M. Kapovich, B. Leeb, and J. Porti, Dynamics on flag manifolds: domains of proper discontinuity and cocompactness, \emph{Geom. Topol.} \textbf{22} (2018), no.~1, 157--234. \href{https://doi.org/10.2140/gt.2018.22.157}{doi:10.2140/gt.2018.22.157}.

\bibitem{KLP25} M. Kapovich, B. Leeb, and J. Porti, Morse actions of discrete groups on symmetric spaces: local-to-global principle, \emph{Geom. Topol.} \textbf{29} (2025), no.~5, 2343--2390. \href{https://doi.org/10.2140/gt.2025.29.2343}{doi:10.2140/gt.2025.29.2343}.



\bibitem{Kob98} S. Kobayashi, \emph{Hyperbolic Complex Spaces}, Grundlehren Math. Wiss., vol.~318, Springer, Berlin, 1998. \href{https://doi.org/10.1007/978-3-662-03582-5}{doi:10.1007/978-3-662-03582-5}.

\bibitem{Kul78} R. S. Kulkarni, Groups with domains of discontinuity, \emph{Math. Ann.} \textbf{237} (1978), no.~3, 253--272. \href{https://doi.org/10.1007/BF01420180}{doi:10.1007/BF01420180}.



\bibitem{Malcev56} A. I. Mal'cev, On certain classes of infinite soluble groups, \emph{Amer. Math. Soc. Transl.} (2) \textbf{2} (1956), 1--21; originally published in \emph{Mat. Sb.} \textbf{28} (1951), 567--588.

\bibitem{McMullen99} C. T. McMullen, Hausdorff dimension and conformal dynamics. I. Strong convergence of Kleinian groups, \emph{J. Differential Geom.} \textbf{51} (1999), no.~3, 471--515. \href{https://doi.org/10.4310/jdg/1214425139}{doi:10.4310/jdg/1214425139}.

\bibitem{Myrberg25} P. J. Myrberg, Untersuchungen über die automorphen Funktionen beliebig vieler Variablen, \emph{Acta Math.} \textbf{46} (1925), no.~3--4, 215--336. \href{https://doi.org/10.1007/BF02564065}{doi:10.1007/BF02564065}.

\bibitem{Wu67} H. Wu, Normal families of holomorphic mappings, \emph{Acta Math.} \textbf{119} (1967), 193--233. \href{https://doi.org/10.1007/BF02392083}{doi:10.1007/BF02392083}.

\bibitem{SV01} J. Seade and A. Verjovsky, Actions of discrete groups on complex projective spaces, in \emph{Geometric Topology and Geometry of Manifolds}, Contemp. Math., vol.~269, Amer. Math. Soc., Providence, RI, 2001, pp.~155--178. \href{https://doi.org/10.1090/conm/269/04332}{doi:10.1090/conm/269/04332}.

\bibitem{SV02} J. Seade and A. Verjovsky, Higher dimensional complex Kleinian groups, \emph{Math. Ann.} \textbf{322} (2002), 279--300. \href{https://doi.org/10.1007/s002080100247}{doi:10.1007/s002080100247}.

\bibitem{SV03} J. Seade and A. Verjovsky, Complex Schottky groups, in \emph{Geometric Methods in Dynamics. II}, Astérisque, vol.~287, Soc. Math. France, Paris, 2003, pp.~251--272. \href{https://doi.org/10.24033/ast.597}{doi:10.24033/ast.597}; \href{https://www.numdam.org/item/AST_2003__287__251_0/}{Numdam}.

\end{thebibliography}
\end{document}